\documentclass{amsart}

\usepackage[all]{xy}
\usepackage{tikz}

\theoremstyle{plain}
\newtheorem{thm}{Theorem}[section]
\newtheorem{prop}[thm]{Proposition}
\newtheorem{lemma}[thm]{Lemma}

\newtheorem{cor}[thm]{Corollary}

\theoremstyle{definition}
\newtheorem{dfn}[thm]{Definition}

\theoremstyle{remark}
\newtheorem{rem}[thm]{Remark}

\newcommand{\HH}{\mathrm{H}}

\begin{document}

\title{Universal deformation rings and self-injective Nakayama algebras}

\author{Frauke M. Bleher}
\address{F.B.: Department of Mathematics\\University of Iowa\\
14 MacLean Hall\\Iowa City, IA 52242-1419, U.S.A.}
\email{frauke-bleher@uiowa.edu}
\thanks{The first author was supported in part by NSF Grant DMS-1360621.}
\author{Daniel J. Wackwitz}
\address{D.W.: Department of Mathematics\\University of Wisconsin-Platteville\\435 Gardner Hall,
1 University Plaza\\Platteville, WI 53818, U.S.A.}
\email{wackwitzd@uwplatt.edu}


\begin{abstract}
Let $k$ be a field and let $\Lambda$ be an indecomposable finite dimensional  $k$-algebra such that there is
a stable equivalence of Morita type between $\Lambda$ and a self-injective split basic Nakayama
algebra over $k$. We show that every indecomposable finitely generated $\Lambda$-module $V$
has a universal deformation ring $R(\Lambda,V)$ and we describe $R(\Lambda,V)$ explicitly as a 
quotient ring of a power series ring over $k$ in finitely many variables.
This result applies in particular to Brauer tree algebras, and hence to $p$-modular blocks of finite
groups with cyclic defect groups.
\end{abstract}

\maketitle

%%%%%%%%%%%%%%%%%%%%%%%%%%%%%%%%%%%%%%%%%%%%%%%%%%%%%%%%%
%% Introduction
%%%%%%%%%%%%%%%%%%%%%%%%%%%%%%%%%%%%%%%%%%%%%%%%%%%%%%%%%

\section{Introduction}
\label{s:intro}
\setcounter{equation}{0}

Let $k$ be a field of arbitrary characteristic, and let $\Lambda$
be a finite dimensional algebra over $k$. Given a finitely generated $\Lambda$-module
$V$, it is a natural question to ask over which complete local commutative Noetherian 
$k$-algebras $R$ with residue field $k$ the module $V$ can be lifted. Here a lift is 
a pair $(M,\phi)$ where $M$ is a free $R$-module with a $\Lambda$-module action and
$\phi:k\otimes_R M\to V$ is a $\Lambda$-module isomorphism. 
It was shown in \cite[Prop. 2.1]{blehervelez} that there exists a particular complete 
local commutative Noetherian $k$-algebra $R(\Lambda,V)$ with residue field $k$ and a 
particular lift $(U,\phi_U)$ of $V$ over $R(\Lambda,V)$ with the following property:
Every lift $(M,\phi)$ of
$V$ over a $k$-algebra $R$ as above is isomorphic to a specialization of $(U,\phi_U)$
via a (not necessarily unique) $k$-algebra homomorphism $R(\Lambda,V)\xrightarrow{\alpha} R$. 
Moreover, $\alpha$ is unique when $R=k[\epsilon]$ is the ring of dual numbers with $\epsilon^2=0$.
The ring $R(\Lambda,V)$ is called a versal deformation ring of $V$ and the isomorphism
class of the lift $(U,\phi_U)$ is called a versal deformation of $V$. 
One is especially interested in the situation when $\alpha$ is {unique} for
every isomorphism class of lifts of $V$ over \emph{every} $k$-algebra $R$ as above, and one calls
$R(\Lambda,V)$ a universal deformation ring of $V$ in this case.
It was shown in \cite[Thm. 2.6]{blehervelez} that when $\Lambda$ is self-injective and 
the stable endomorphism ring of $V$ over $\Lambda$ is isomorphic to $k$, then $R(\Lambda,V)$ 
is universal.
The question remains for which algebras $\Lambda$ every finitely generated indecomposable
non-projective $\Lambda$-module has a universal deformation ring.

In this paper, we study the case when $\Lambda$ is an indecomposable $k$-algebra that is
stably Morita equivalent to a self-injective split basic 
Nakayama algebra and $V$ is an arbitrary finitely generated indecomposable non-projective $\Lambda$-module.
Our main goal is to show that no matter how big the $k$-dimension of the stable endomorphism ring 
of $V$ is, $V$ always has a universal deformation ring. Moreover, we will give an 
explicit description of this universal deformation ring for each such $V$ in terms of generators and relations
that only depends on the location of $[V]$ in the stable Auslander-Reiten quiver of $\Lambda$.

Before stating our results, let us discuss some background on studying lifts and deformation rings of modules.

The problem of lifting modules has a long tradition when $\Lambda$ is replaced by the group ring $kG$ of a finite 
(or profinite) group 
$G$ and $k$ is a perfect field of positive characteristic $p$. In this case, one not only studies lifts of $V$ to 
complete local commutative Noetherian $k$-algebras but to arbitrary complete local commutative 
Noetherian rings with residue field $k$. 
One of the first results in this direction is due to Green who proved in \cite{Green} that if $k$ is the residue field 
of a ring of $p$-adic integers $\mathcal{O}$ then a finitely generated $kG$-module $V$ can be lifted to 
$\mathcal{O}$ if there are no non-trivial 2-extensions of $V$ by itself. 
Green's work inspired Auslander, Ding and Solberg in \cite{ads} to
consider more general  algebras over Noetherian rings and more general lifting problems. 
In \cite{rickardtilting}, Rickard generalized Green's result to modules for arbitrary finite rank algebras over
complete local commutative Noetherian rings, as a consequence of his study of lifts of tilting complexes.
On the other hand, Laudal developed a theory of formal moduli of algebraic structures, 
and, working over an arbitrary field $k$,  he used Massey products
to describe deformations of $k$-algebras and their modules over complete local commutative Artinian $k$-algebras
with residue field $k$ (see \cite{laudal1983} and its references). 

Sometimes  it may happen
that the algebra whose modules and their deformations one would like to study is 
only known up to a derived or stable equivalence.
In \cite{blehervelezderived}, the behavior of deformations under such equivalences was studied.
In particular, it was shown in \cite[Sect. 3.2]{blehervelezderived} that versal 
deformation rings of modules for self-injective algebras are preserved under stable equivalences of Morita type.
Hence these versal deformation rings provide invariants of such equivalences.

\medskip

In this paper we let $k$ be an arbitrary field, and we concentrate on finite dimensional $k$-algebras of finite 
representation type. More specifically, we focus on indecomposable $k$-algebras 
$\Lambda$ for which there exists a stable equivalence of Morita type to a self-injective split basic Nakayama algebra 
over $k$. 

In \cite{GabrielRiedtmann}, Gabriel and Riedtmann showed that Brauer tree algebras 
are stably equivalent to symmetric split basic Nakayama algebras. Moreover, Rickard proved in 
\cite[Sect. 4]{Rickard1989} that there is a derived equivalence, and hence by \cite[Sect. 5]{Rickard1991}
a stable equivalence of Morita type, between these algebras. 
Since by \cite{BrauerCyclic,DadeCyclic}, a $p$-modular block of a finite group $G$ with cyclic 
defect groups is a Brauer tree algebra (over a field of characteristic $p$ that is sufficiently large for $G$), our 
results apply in particular to this case; see below.

Note that the assumption that $\Lambda$ is indecomposable is no restriction when one considers
deformation rings of finitely generated indecomposable $\Lambda$-modules. This follows, since if $B$ is an
indecomposable direct factor algebra of $\Lambda$ and $V$ is a $\Lambda$-module that belongs to $B$  
then the versal deformation rings of $V$ viewed either as a $B$-module or as a $\Lambda$-module are 
isomorphic (see Lemma \ref{lem:vdrblocks}).

To state our main results, we need the following definition.

\begin{dfn}
\label{def:needthis}
Let $k$ be a field.
\begin{itemize}
\item[(a)] For every positive integer $e$, let $Q_e$ be the circular quiver 
with $e$ vertices, labeled $1,\ldots,e$, and $e$ arrows, labeled $\alpha_1,\ldots,\alpha_e$, such that 
$\alpha_i: i\to i+1$, where the vertex $e+1$ is identified with 1. Let $\mathcal{J}$ be the ideal of the path algebra $k\,Q_e$ 
generated by all arrows, i.e. by all paths of length 1. 
For all integers $e\ge 1$ and $\ell\ge 2$, define $\mathcal{N}(e,\ell)=k\,Q_e/\mathcal{J}^\ell$.

\item[(b)] For any integer $n\ge 1$, let $N_n$ be the $n\times n$ matrix with entries in the power series algebra 
$k[[t_1,\ldots,t_n]]$ defined by
$$N_n=\left(\begin{array}{ccc|c}0&\cdots&0&t_n\\ \hline &&&t_{n-1}\\&\mathrm{I}_{n-1}&&\vdots\\&&&t_1\end{array}\right)$$
where $\mathrm{I}_{n-1}$ is the $(n-1)\times (n-1)$ identity matrix. In particular, $N_1=(t_1)$.
Let $a\ge 0$ be an integer. If $n\ge 1$, define $J_{n}(a)$ to be the ideal of $k[[t_1,\ldots,t_n]]$ generated by the 
entries in $(N_n)^a$. If $n=0$, define $J_{0}(a)$ to be the zero ideal of $k$.
\end{itemize}
\end{dfn}

It is well-known (see, for example, \cite[p. 243]{GabrielRiedtmann})
that every indecomposable self-injective non-semisimple split basic Nakayama algebra over $k$ 
is isomorphic to $\mathcal{N}(e,\ell)$, as in Definition \ref{def:needthis}(a), 
for appropriate integers $e\ge 1$ and $\ell\ge 2$. 

In our first main result we show that the versal deformation ring of each finitely generated indecomposable non-projective
$\mathcal{N}(e,\ell)$-module is universal, and we describe this ring explicitly using the ideals introduced in
Definition \ref{def:needthis}(b).

\begin{thm}
\label{thm:supermain}
Let $k$ be an arbitrary field, let $e\ge 1$ and $\ell\ge 2$ be integers, and let 
$\mathcal{N}(e,\ell)=k\,Q_e/\mathcal{J}^\ell$ be 
as in Definition $\ref{def:needthis}(a)$.
Write
\begin{equation}
\label{eq:ell}
\ell = \mu\, e + \ell'
\end{equation}
where $\mu,\ell'\ge 0$ are integers and $\ell'\le e-1$. Let $V$ be a finitely generated indecomposable non-projective
$\mathcal{N}(e,\ell)$-module, and define $d_V\ge 0$ to be the distance of $[V]$ from the closest
boundary of the stable Auslander-Reiten quiver of $\mathcal{N}(e,\ell)$, where distance $0$ corresponds to
modules at a boundary. In other words, if $\ell_V=\mathrm{min}\{\mathrm{dim}_k\,V,\ell-\mathrm{dim}_k\,V\}$
then $\ell_V=d_V+1$. Write $\ell_V$ as
\begin{equation}
\label{eq:ellV}
\ell_V = n\, e + i
\end{equation}
where $n,i\ge 0$ are integers with $i\le e-1$. 

The versal deformation ring $R(\mathcal{N}(e,\ell),V)$ is universal.
Moreover, 
$$R(\mathcal{N}(e,\ell),V)\cong k[[t_1,\ldots,t_n]]/J_{n}(m_V)$$
where 
\begin{equation}
\label{eq:mV}
m_V=\left\{\begin{array}{ccl}
\mu&:&0\le i \le \ell',\\
\mu-1&:&\ell'+1\le i \le e-1.\end{array}\right.
\end{equation}
and $J_{n}(m_V)$ is as in Definition $\ref{def:needthis}(b)$.
\end{thm}

Our second main result shows that Theorem \ref{thm:supermain} can be generalized to
indecomposable finite dimensional $k$-algebras $\Lambda$ for which there exists
a stable equivalence of Morita type to a self-injective split basic Nakayama algebra.

\begin{thm}
\label{thm:main2}
Let $k$ be an arbitrary field, and let $\Lambda$ be an indecomposable finite dimensional  $k$-algebra such that there exists
a stable equivalence of Morita type between $\Lambda$ and a  self-injective split basic Nakayama
algebra $\mathcal{N}$ over $k$. Suppose $V$ is a finitely generated indecomposable $\Lambda$-module.

The versal deformation ring $R(\Lambda,V)$ is universal and has the following isomorphism type:

\begin{itemize}
\item[(i)] If $V$ is projective then $R(\Lambda,V)\cong k$.
\item[(ii)] Suppose $V$ is not projective and $\Lambda$ has Loewy length $L=2$. Then $V$ is a simple 
non-projective $\Lambda$-module. If $\mathrm{Ext}^1_\Lambda(V,V)=0$ then $R(\Lambda,V)\cong k$.
If $\mathrm{Ext}^1_\Lambda(V,V)\ne 0$ then $R(\Lambda,V)\cong k[[t]]/(t^2)$.
\item[(iii)] Suppose $V$ is not projective and $\Lambda$ has Loewy length $L\ge 3$. 
Then there exist integers $e\ge 1$ and $\ell\ge 3$ such that $\mathcal{N}\cong\mathcal{N}(e,\ell)$.
Write $\ell = \mu\, e + \ell'$ as in $(\ref{eq:ell})$.
Suppose $[V]$ has distance $d_V\ge 0$ from the closest
boundary of the stable Auslander-Reiten quiver $\Gamma_s(\Lambda)$, where distance $0$ corresponds to
modules at a boundary. Define $\ell_V=d_V+1$.
Write $\ell_V= n\, e + i$ as in $(\ref{eq:ellV})$, and define $m_V$ as in $(\ref{eq:mV})$.
Then $R(\Lambda,V)\cong k[[t_1,\ldots,t_n]]/J_n(m_V)$.
\end{itemize}
\end{thm}

For Brauer tree algebras, and hence in particular for $p$-modular blocks of finite groups with cyclic defect groups,
we obtain the following consequence.

\begin{cor}
\label{cor:supermain}
Let $k$ be an arbitrary field, and let $\Lambda$ be a Brauer tree algebra whose Brauer tree has $e$ edges and
an exceptional vertex of multiplicity $m\ge 1$. 

Suppose $V$ is a finitely generated indecomposable non-projective $\Lambda$-module, 
and suppose $[V]$ has distance $d_V\ge 0$ from the closest
boundary of the stable Auslander-Reiten quiver $\Gamma_s(\Lambda)$. Define $\ell_V=d_V+1$, and write $\ell_V= n\, e + i$ 
where $n,i\ge 0$ are integers with $i\le e-1$.
Define 
$$m_V=\left\{\begin{array}{ccl}
m+1&:&e=1,\\
m&:&e>1\mbox{ and } i\in\{0,1\},\\
m-1&:&e>1 \mbox{ and }2\le i \le e-1.\end{array}\right.$$
Then $R(\Lambda,V)$ is universal and isomorphic to $k[[t_1,\ldots,t_n]]/J_n(m_V)$.
\end{cor}

Note that the case when $\Lambda$ is a Brauer tree algebra and $V$ is a finitely generated $\Lambda$-module
whose stable endomorphism ring is isomorphic to $k$ already follows from the results and methods in \cite{bc}. 

\medskip

Let us now outline the organization of the paper and summarize the main ideas of the proofs of our main results.

In Section \ref{s:prelim}, we give an introduction to versal and universal deformation rings and deformations
of finitely generated modules $V$ for a finite dimensional $k$-algebra $\Lambda$. In particular, we show 
in Lemma \ref{lem:vdrblocks} that if $B$ is an indecomposable direct factor algebra, i.e. a block, of $\Lambda$ 
and $V$ is a $B$-module then the versal deformation rings of $V$ viewed either as a $B$-module or
as a $\Lambda$-module are isomorphic. 
In Proposition \ref{prop:generalOmega}, we prove that if $\Lambda$ is a Frobenius algebra then the first syzygy functor
preserves the versal deformation ring of an arbitrary non-projective finitely generated $\Lambda$-module, 
generalizing a result in \cite{blehervelez}. 

In Section \ref{s:nakayama}, we prove Theorem \ref{thm:supermain}, using the following key steps. 
Suppose $V$ is a finitely generated indecomposable non-projective $\mathcal{N}(e,\ell)$-module. 
By replacing $V$ by $\Omega(V)$, if necessary, we can assume that $\ell_V=\mathrm{dim}_k\,V$. By taking a cyclic permutation of
the vertices $1,\ldots, e$ of the quiver $Q_e$ of $\mathcal{N}(e,\ell)$, if necessary, we can also assume
that the radical quotient of $V$ is isomorphic to the simple $\mathcal{N}(e,\ell)$-module corresponding to the vertex $1$.
Write $\ell_V= n\, e + i$ as in $(\ref{eq:ellV})$, and define $m_V$ as in $(\ref{eq:mV})$. 
We first prove that $\mathrm{Ext}^1_{\mathcal{N}(e,\ell)}(V,V)$ is an $n$-dimensional vector space over $k$
(see Lemma \ref{lem:ext1}) and provide an explicit $k$-basis for $\mathrm{Ext}^1_{\mathcal{N}(e,\ell)}(V,V)$ in terms of
extensions of $V$ by itself (see Lemma \ref{lem:ext2}). We then use this to define a lift of $V$ over the ring 
$R_V=k[[t_1,\ldots,t_n]]/J_n(m_V)$, by providing an explicit matrix representation $\rho_U:\mathcal{N}(e,\ell)\to
\mathrm{Mat}_{\ell_V}(R_V)$ (see Lemma \ref{lem:ideal!!!}). In Theorem \ref{thm:versallift}, we then show that 
$R_V$ is isomorphic to the versal deformation ring $R(\mathcal{N}(e,\ell),V)$ and that $\rho_U$
defines a versal lift of $V$ over $R_V$. 
Finally, in Theorem \ref{thm:universal}, we show that $R(\mathcal{N}(e,\ell),V)$ is universal by proving that the deformation
functor associated to $V$ has the centralizer lifting property (see Definition \ref{def:centralizerlifting} and Lemma 
\ref{lem:centralizerlifting}).

In Section \ref{s:stableMorita}, we first review stable equivalences of Morita type. We then prove Theorem \ref{thm:main2} and 
Corollary \ref{cor:supermain}. For the proof of Theorem \ref{thm:main2}, one of the main ingredients is Reiten's characterization  in \cite{Reiten}
of Artin algebras that are stably equivalent to self-injective algebras. Moreover, we use the results in 
\cite[Sect. 3.2]{blehervelezderived}. For the proof of Corollary \ref{cor:supermain}, we moreover use \cite[Sect. 4]{Rickard1989}
and \cite[Sect. 5]{Rickard1991}.

Part of this paper constitutes the Ph.D. thesis  \cite{wackwitz} of the second author under the supervision of the first author.

\medskip

Unless specifically stated otherwise, all our modules are assumed to be 
finitely generated left modules. In fact, right modules only occur in Remark \ref{rem:josepaper}
and the proof of Proposition \ref{prop:generalOmega} when considering dual modules, in
addition to Section \ref{s:stableMorita} where they occur in the context of bimodules.
We write maps on the left so that the map composition $f\circ g$ means that we apply $f$ after $g$.

%%%%%%%%%%%%%%%%%%%%%%%%%%%%%%%%%%%%%%%%%%%%%%%%%%%%%%%%%
%% Versal and universal deformation rings
%%%%%%%%%%%%%%%%%%%%%%%%%%%%%%%%%%%%%%%%%%%%%%%%%%%%%%%%%

\section{Versal and universal deformation rings}
\label{s:prelim}
\setcounter{equation}{0}

Let $k$ be a field of arbitrary characteristic. Let $\hat{\mathcal{C}}$ be the category of 
all complete local commutative Noetherian $k$-algebras $R$ with residue field $k$. For
each such $R$, let $\pi_R:R\to k$ be the corresponding reduction map and let $\mathfrak{m}_R$ denote its unique maximal ideal. The morphisms in 
$\hat{\mathcal{C}}$ are continuous $k$-algebra homomorphisms which induce the identity map on $k$. 
Let $\mathcal{C}$ be the full subcategory of $\hat{\mathcal{C}}$ consisting of Artinian rings.

Suppose $\Lambda$ is a finite dimensional $k$-algebra and $V$ is a finitely generated
$\Lambda$-module. Let $R$ be a ring in $\hat{\mathcal{C}}$, and define $R\Lambda=R\otimes_k\Lambda$.
A \emph{lift} of $V$ over $R$ is a finitely generated $R\Lambda$-module $M$ which
is free over $R$ together with a $\Lambda$-module isomorphism $\phi:k\otimes_R M\to V$. Two lifts 
$(M,\phi)$ and $(M',\phi')$ of $V$ over $R$ are \emph{isomorphic} if there exists an 
$R\Lambda$-module isomorphism  $f:M\to M'$ such that 
$\phi'\circ(k\otimes_R f) = \phi$. The isomorphism class of a lift $(M,\phi)$ of $V$ 
over $R$ is denoted by $[M,\phi]$ and called a \emph{deformation} of $V$ over $R$. 
We define $\mathrm{Def}_{\Lambda}(V,R)$ to be the set of all deformations of $V$ over $R$. 
If $\alpha:R\to R'$ is a morphism in $\hat{\mathcal{C}}$,  we define a map
$$\xymatrix @R=.2pc{
\mathrm{Def}_\Lambda(V,\alpha):&\mathrm{Def}_{\Lambda}(V,R)\ar[r]&\mathrm{Def}_{\Lambda}(V,R')\\
&[M,\phi]\ar@{|->}[r]& [R'\otimes_{R,\alpha}M,\phi_\alpha]}$$
where $\phi_\alpha$ is the composition of $\Lambda$-module homomorphisms
$$k\otimes_{R'}(R'\otimes_{R,\alpha}M) \cong k\otimes_R M\xrightarrow{\phi} V\;.$$
With these definitions $\mathrm{Def}_{\Lambda}(V,-)$ is a 
covariant functor from $\hat{\mathcal{C}}$ to the category of sets.

Alternatively, we can describe $\mathrm{Def}_{\Lambda}(V,-)$ using matrices as follows.
Suppose $\mathrm{dim}_k\,V = n$. By choosing a $k$-basis of $V$ we can identify $V$ with 
$k^n$ and $\mathrm{End}_k(V)$ with $\mathrm{Mat}_n(k)$. The action of $\Lambda$ on $V$ is
then given by a $k$-algebra homomorphism $\rho_V:\Lambda\to \mathrm{Mat}_n(k)$.
Let $R$ be a ring in $\hat{\mathcal{C}}$ and denote the reduction map
$\mathrm{Mat}_n(R)\to \mathrm{Mat}_n(k)$ (resp. $\mathrm{GL}_n(R)\to \mathrm{GL}_n(k)$)
also by $\pi_R$. A \emph{lift} of $\rho_V$ over 
$R$ is a $k$-algebra homomorphism $\tau:\Lambda\to \mathrm{Mat}_n(R)$ such that
$\pi_R\circ\tau = \rho_V$. Since $R\Lambda=R\otimes_k\Lambda$, 
such a lift defines an $R\Lambda$-module action on $M=R^n$,
and with the obvious isomorphism $\phi:k\otimes_RM\to V$ such a lift defines a deformation
$[M,\phi]$ of $V$ over $R$.
Two lifts $\tau,\tau':\Lambda\to \mathrm{Mat}_n(R)$ of $\rho_V$ over $R$ give rise
to the same deformation if and only if they are 
\emph{strictly equivalent} in the sense that there exists an element 
$C$ in the kernel of $\mathrm{GL}_n(R)\xrightarrow{\pi_R} \mathrm{GL}_n(k)$
such that $C\,\tau\,C^{-1}=\tau'$. Denote the strict equivalence class
of $\tau$ by $[\tau]$ and define $\mathrm{Def}_{\Lambda}(\rho_V,R)$ to be the set of all 
strict equivalence classes of lifts of $\rho_V$ over $R$. 
In this way, the choice of a $k$-basis of $V$ gives rise to an identification of
$\mathrm{Def}_{\Lambda}(V,R)$ with $\mathrm{Def}_{\Lambda}(\rho_V,R)$.
In the following, we identify the two functors $\mathrm{Def}_{\Lambda}(V,-)=
\mathrm{Def}_{\Lambda}(\rho_V,-)$.

Let $k[\epsilon]$, where $\epsilon^2=0$, denote the ring of dual numbers over $k$. 
The tangent space of $\mathrm{Def}_{\Lambda}(V,-)$ is defined to be the set
$t_V=\mathrm{Def}_{\Lambda}(V,k[\epsilon])$.

By \cite[Prop. 2.1]{blehervelez}, there is a $k$-vector space isomorphism 
$t_V \cong\mathrm{Ext}^1_{\Lambda}(V,V)$, and the restriction of the functor
$\mathrm{Def}_{\Lambda}(V,-)$ to $\mathcal{C}$ has a pro-representable hull 
$R(\Lambda,V)$ in $\hat{\mathcal{C}}$. This means that there exists 
a deformation $[U(\Lambda,V),\phi_U]$ of $V$ over $R(\Lambda,V)$ 
with the following properties. For each ring $R$ in $\hat{\mathcal{C}}$, the map
$\mathrm{Hom}_{\hat{\mathcal{C}}}(R(\Lambda,V),R) \to \mathrm{Def}_{\Lambda}(V,R)$ 
given by $\alpha \mapsto \mathrm{Def}_{\Lambda}(V,\alpha)([U(\Lambda,V),\phi_U])$ is surjective,
and this map is bijective if $R$ is the ring of dual numbers $k[\epsilon]$.
In particular, this implies that if $\mathrm{dim}_k\,\mathrm{Ext}^1_{\Lambda}(V,V)=r$
then $R(\Lambda,V)$ is isomorphic to a quotient algebra of the power series ring
$k[[t_1,\ldots,t_r]]$, and $r$ is minimal with this
property.

The ring $R(\Lambda,V)$  is called the \emph{versal deformation ring} of $V$ and
$[U(\Lambda,V),\phi_U]$ is called the \emph{versal deformation} of $V$. In general,
the isomorphism type of $R(\Lambda,V)$ is unique up to a (non-canonical) isomorphism.

If $R(\Lambda,V)$ represents $\mathrm{Def}_{\Lambda}(V,-)$, then $R(\Lambda,V)$  is called 
the \emph{universal deformation ring} of $V$ and $[U(\Lambda,V),\phi_U]$ is called the 
\emph{universal deformation} of $V$. In this case, $R(\Lambda,V)$ is unique up to 
a canonical isomorphism. 

By \cite[Prop. 2.1]{blehervelez}, $R(\Lambda,V)$ is always universal 
if the endomorphism ring $\mathrm{End}_\Lambda(V)$ is isomorphic to $k$. The following
easy remark gives another example of a universal deformation ring.

\begin{rem}
\label{rem:superlame}
Suppose $V$ is a finitely generated non-zero $\Lambda$-module such that 
$\mathrm{Ext}^1_\Lambda(V,V)=0$. Then the versal deformation ring $R(\Lambda,V)$ is isomorphic to $k$. 
For each $R\in\mathrm{Ob}(\hat{\mathcal{C}})$, let $\iota_R:k\to R$ be the unique morphism in 
$\hat{\mathcal{C}}$ giving $R$ its $k$-algebra structure. Then 
$\mathrm{Hom}_{\hat{\mathcal{C}}}(k,R)=\{\iota_R\}$, which implies that $k$ is the universal deformation ring of $V$.

In particular, if $P$ is a finitely generated non-zero projective $\Lambda$-module, then the versal deformation
ring $R(\Lambda,P)$ is universal and isomorphic to $k$. 
\end{rem}

Note that $\Lambda$ decomposes into a direct product of indecomposable direct factor algebras, 
also called blocks:
$$\Lambda = B_1\times \cdots\times B_r\,.$$
These blocks correspond to a decomposition $1_\Lambda=e_1+\cdots+e_r$ into a sum of
pairwise orthogonal primitive central idempotents. Recall that this decomposition is unique up to permutation
of the summands. For $1\le i\le r$, we call $e_i$ the block idempotent of $B_i$.
The following result shows that if $B$ is a block of $\Lambda$ and $V$ is a $\Lambda$-module
belonging to $B$, then one can restrict oneself to $B$ when computing the (uni-)versal deformation 
ring $R(\Lambda,V)$. Note that we use $\mathrm{rad}(A)$ to denote the Jacobson radical of a ring $A$.

\begin{lemma}
\label{lem:vdrblocks}
Suppose $B$ is a block of $\Lambda$ with block idempotent $e_B$, and suppose $V$ is a $\Lambda$-module belonging to $B$.
Then $R(\Lambda,V)\cong R(B,V)$ in $\hat{\mathcal{C}}$. Moreover, $R(\Lambda,V)$ is universal if and only if $R(B,V)$ is universal.
\end{lemma}

\begin{proof}
We fix a ring $R$ in $\hat{\mathcal{C}}$ and let $\mathfrak{m}_R$ be its unique maximal ideal. Define $N$ to be the ideal of $R\Lambda$
generated by $\mathfrak{m}_R$, i.e. $N=\mathfrak{m}_R\Lambda$. By \cite[Props. 5.22 and 6.5]{CR}, $N\subseteq\mathrm{rad}(R\Lambda)$ and
the $R$-algebra $R\Lambda$ is complete in the 
$N$-adic topology. By \cite[Thm. 22.11]{Lam}, we can therefore use the natural map 
\begin{equation}
\label{eq:natural}
R\Lambda\to R\Lambda/N=R\Lambda/\mathfrak{m}_R\Lambda\cong \Lambda
\end{equation}
to lift the primitive central idempotent $e_B$ in $\Lambda$ to a primitive central idempotent $\hat{e}_B$ in $R\Lambda$. 
On the other hand, $1_R\otimes e_B$ is a central idempotent of $R\Lambda=R\otimes_k\Lambda$ that is sent by the map 
(\ref{eq:natural}) to $e_B$. Since by \cite[Thm. 22.11]{Lam}, $1_R\otimes e_B$ is centrally primitive in $R\Lambda$
if and only if its image under the map (\ref{eq:natural}) is centrally primitive in $\Lambda$ and since the primitive
central idempotents of $R\Lambda$ are unique, it follows that $\hat{e}_B=1_R\otimes e_B$. In particular, we have
$$\hat{e}_B(R\Lambda)\hat{e}_B = (1_R\otimes e_B)(R\otimes_k\Lambda)(1_R\otimes e_B) = R\otimes_k (e_B\Lambda e_B) = R(e_B\Lambda e_B)=RB\,.$$
We therefore obtain a well-defined map
\begin{equation}
\label{eq:well}
\xymatrix @R=.2pc{
g_R:&\mathrm{Def}_{B}(V,R)\ar[r]&\mathrm{Def}_{\Lambda}(V,R)\\
&[M,\phi]\ar@{|->}[r]& [M,\phi]}
\end{equation}
which is natural with respect to morphisms $\alpha:R\to R'$ in $\hat{\mathcal{C}}$. Moreover, $g_R$ is injective since
every $R\Lambda$-module homomorphism between $R\Lambda$-modules belonging to
$\hat{e}_B(R\Lambda)\hat{e}_B = RB$ is in particular an $RB$-module homomorphism.

To show that $g_R$ is surjective, let $(M,\phi)$ be a lift of $V$ over $R$ when $V$ is viewed as a $\Lambda$-module. Then 
$M= \hat{e}_BM + (1-\hat{e}_B)M$. Since $(1-e_B)V=0$, it follows that $(1-\hat{e}_B)M\subseteq \mathfrak{m}_R M \subseteq \mathrm{rad}(R\Lambda)\,M$. Since $M$ is a finitely generated 
$R\Lambda$-module, we therefore obtain by Nakayama's Lemma that $M = \hat{e}_B M$.
But this means that $M$ is an $RB$-module, and hence $(M,\phi)$ is a lift of $V$ over $R$ when $V$ is viewed as a $B$-module.

It follows that the bijections $g_R$, for $R$ in $\hat{\mathcal{C}}$, define a natural isomorphism
between the deformation functors $\mathrm{Def}_{B}(V,-)$ and 
$\mathrm{Def}_{\Lambda}(V,-)$.
\end{proof}

By \cite[Thm. 2.6]{blehervelez}, if $\Lambda$ is self-injective 
and the stable endomorphism ring $\underline{\mathrm{End}}_\Lambda(V)$ is isomorphic
to $k$, then $R(\Lambda,V)$ represents $\mathrm{Def}_{\Lambda}(V,-)$ and hence $R(\Lambda,V)$ is a 
universal deformation ring of $V$. Moreover, if $\Lambda$ is a Frobenius algebra and 
$\underline{\mathrm{End}}_\Lambda(V)\cong k$
then $R(\Lambda,\Omega(V))\cong R(\Lambda,V)$ in $\hat{\mathcal{C}}$,
where $\Omega(V)$ is the first syzygy of $V$. In other words, $\Omega(V)$ is the kernel of a 
projective cover  $\psi_V:P(V)\to V$. 

We next want to prove that if $\Lambda$ is a Frobenius algebra then 
we always have that the versal deformation rings $R(\Lambda,V)$ and $R(\Lambda,\Omega(V))$
are isomorphic, even if $\underline{\mathrm{End}}_\Lambda(V)$ is not isomorphic to $k$. The proof is 
considerably more involved in this general case, since we cannot assume that 
$R(\Lambda,V)$ represents $\mathrm{Def}_{\Lambda}(V,-)$.

We first collect some useful facts in Remark \ref{rem:josepaper}, which were proved as Claims 1 and 7
in the proof of \cite[Thm. 2.6]{blehervelez} without any assumption on $\underline{\mathrm{End}}_\Lambda(V)$. 
Note that we need to add the assumption that $M$ (resp. $M_0$) is free over $R$ (resp. $R_0$) in Claims 1 and 2
in the proof of \cite[Thm. 2.6]{blehervelez}. This makes no difference since these claims were only used under this assumption.
As before, $\mathcal{C}$ denotes the full 
subcategory of $\hat{\mathcal{C}}$  consisting of Artinian rings.

\begin{rem}
\label{rem:josepaper}
Suppose $\Lambda$ is a self-injective finite dimensional $k$-algebra.
Let $R,R_0$ be in $\mathcal{C}$, and let $\pi:R\to R_0$ be a surjection in $\mathcal{C}$.
Let $M$, $Q$ (resp. $M_0$, $Q_0$) be finitely generated $R\Lambda$-modules
(resp. $R_0\Lambda$-modules) that are free over $R$ (resp. $R_0$),
and assume that $Q$ (resp. $Q_0$) is projective. 
Suppose there are
$R_0\Lambda$-module isomorphisms $g:R_0\otimes_{R,\pi} M \to M_0$ and
$h:R_0\otimes_{R,\pi} Q\to Q_0$. 
\begin{enumerate}
\item[(i)] 
If $\nu_0\in \mathrm{Hom}_{R_0\Lambda}(M_0, Q_0)$, then there
exists $\nu\in \mathrm{Hom}_{R\Lambda}(M, Q)$ with $\nu_0=
h\circ(R_0\otimes_{R,\pi}\nu)\circ g^{-1}$.
\item[(ii)]
Suppose $\Lambda$ is a Frobenius algebra, and
$P$ is a finitely generated projective left (resp. right) $R\Lambda$-module. 
Then $P^* = \mathrm{Hom}_R(P , R)$ is a projective right (resp. left)
$R\Lambda$-module.
\end{enumerate}
\end{rem}

\begin{prop}
\label{prop:generalOmega}
Let $\Lambda$ be a Frobenius algebra, and let $V$ be a non-projective finitely generated 
$\Lambda$-module. Then $R(\Lambda, V) \cong R(\Lambda, \Omega(V))$
in $\hat{\mathcal{C}}$. Moreover, $R(\Lambda,\Omega(V))$ is universal if and only if $R(\Lambda,V)$ is universal.
\end{prop}

\begin{proof}
By \cite[Lemma 3.2.2]{blehervelezderived}, we have 
$R(\Lambda,V)\cong R(\Lambda, V\oplus Q)$ 
for any finitely generated projective
$\Lambda$-module $Q$. Therefore, we assume now for the remainder of the proof that
$V$ has no projective direct summands.

Since $\Lambda$ is a Frobenius algebra, we have
${}_\Lambda\Lambda \cong (\Lambda_\Lambda)^*=
\mathrm{Hom}_k(\Lambda_\Lambda,k)$ as left $\Lambda$-modules
and also $\Lambda_\Lambda \cong ({}_\Lambda\Lambda)^*=
\mathrm{Hom}_k({}_\Lambda\Lambda,k)$ as right $\Lambda$-modules.
This implies, in particular, that $\Lambda$ is injective both as a left and as a 
right module over itself.

Let $T$ be a finitely generated non-zero left (resp. right) 
$\Lambda$-module such that $T$ has no projective direct summands.
We fix a projective cover $\psi_T:P(T)\to T$  of $T$, which means that 
$P(T)$ is a left (resp. right) projective $\Lambda$-module and
$\psi_T$ is an essential epimorphism. Then $\Omega(T)$ is the kernel of
$\psi_T$ and we have a short exact sequence of  left (resp. right)
$\Lambda$-modules
\begin{equation}
\label{eq:syzygy1}
\xymatrix{
0\ar[r]& \Omega(T)\ar[r]^{\iota_T}&P(T)\ar[r]^{\psi_T}&T\ar[r]&0}
\end{equation}
where $\iota_T$ is the inclusion map. 
Since  $\Lambda$ is self-injective, it follows that 
$P(T)$ is also an injective $\Lambda$-module. Note that since we assume 
$T$ has no projective direct summands, the same is true for $\Omega(T)$. 
This means that if we apply $\mathrm{Hom}_k(-,k)$ to the sequence
(\ref{eq:syzygy1}), we obtain a short exact sequence of right (resp. left)
$\Lambda$-modules
\begin{equation}
\label{eq:syzygy2}
\xymatrix{
0\ar[r]& T^*\ar[r]^{\psi_T^*}&P(T)^*\ar[r]^{\iota_T^*}&\Omega(T)^*\ar[r]&0}
\end{equation}
where $P(T)^*$ is a projective right $\Lambda$-module. Moreover, since
$\Omega(T)^*$ and $T^*$ do not have projective direct summands, it
follows that $\iota_T^*:P(T)^*\to \Omega(T)^*$ is a projective cover
of $\Omega(T)^*$. 

Fix an Artinian ring $R$ in $\mathcal{C}$. Let $M$ be a finitely generated
left (resp. right) $R\Lambda$-module that is free as an $R$-module.
Define $M^*=\mathrm{Hom}_R(M,R)$ where the right (resp. left) 
$R\Lambda$-module 
structure is induced by the left (resp. right) $R\Lambda$-module 
structure of $M$.
Define $(k\otimes_RM)^*=\mathrm{Hom}_k(k\otimes_R M,k)$.
There is a $\Lambda$-module isomorphism 
$k\otimes_RM^*\cong (k\otimes_RM)^*$
which is natural with respect to homomorphisms between finitely generated
$R\Lambda$-modules that are free as $R$-modules.

We first prove two claims.

\medskip

\noindent
\textit{Claim $1$}. 
Let $X$ be a finitely generated left (resp. right) $R\Lambda$-module 
that is free as an $R$-module such that there is a $\Lambda$-module isomorphism
$\xi:k\otimes_RX\to T$. Assume that there
is a short exact sequence of left (resp. right) $R\Lambda$-modules
\begin{equation}
\label{eq:ses1}
0\to Y\xrightarrow{\iota} P\xrightarrow{\psi}X\to 0
\end{equation}
such that $P$ is a projective left (resp. right) $R\Lambda$-module with 
$k\otimes_RP\cong P(T)$. Then $\psi$ is an essential epimorphism,
implying that $\psi: P\to X$ is a projective $R\Lambda$-module cover of $X$. 

\medskip

\noindent
\textit{Proof of Claim $1$}. 
Since $k\otimes_RP$ is projective, there exist $\Lambda$-module homomorphisms
$\overline{\lambda}:k\otimes_RP \to P(T)$ and $\overline{\mu}:k\otimes_R Y\to \Omega(T)$ making the following diagram of left (resp. right)
$\Lambda$-modules commutative:
\begin{equation}
\label{eq:first}
\xymatrix{
0\ar[r]&k\otimes_R Y\ar[r]^{k\otimes\iota}\ar[d]_{\overline{\mu}}& k\otimes_RP\ar[r]^{k\otimes\psi} \ar[d]_{\overline{\lambda}}&k\otimes_RX\ar[r]\ar[d]^{\xi}&0\\
0\ar[r]&\Omega(T)\ar[r]^{\iota_T} & P(T)\ar[r]^{\psi_T}&T\ar[r]&0.
}
\end{equation}
Since $\psi_T$ is an essential epimorphism, it follows that 
$\overline{\lambda}$ is surjective, and hence bijective
because $k\otimes_RP$ and $P(T)$ have the same $k$-dimension.
This implies that $\overline{\lambda}$ and $\overline{\mu}$ are 
$\Lambda$-module isomorphisms. In particular, this means that
$k\otimes\psi$ is an essential epimorphism. 
But then it follows from Nakayama's Lemma that $\psi$ is also an
essential epimorphism, which proves Claim 1.

\medskip

\noindent
\textit{Claim $2$}. 
For $i=1,2$, let $X_i$ be a finitely generated left (resp. right) $R\Lambda$-module that
is free as an $R$-module such that there is a $\Lambda$-module isomorphism
$\xi_i:k\otimes_RX_i\to T$. Assume that there
is a short exact sequence of left (resp. right) $R\Lambda$-modules
$$0\to Y_i\xrightarrow{\iota_i} P_i\xrightarrow{\psi_i}X_i\to 0$$
such that $P_i$ is a projective left (resp. right) $R\Lambda$-module with $k\otimes_RP_i\cong P(T)$. Suppose $\Xi_i:k\otimes_RP_i\to P(T)$ 
and $\Omega(\xi_i):k\otimes_RY_i\to \Omega(T)$ are
$\Lambda$-module homomorphisms such that there
is a commutative diagram 
\begin{equation}
\label{eq:needthis!}
\xymatrix{
0\ar[r]&k\otimes_RY_i\ar[r]^{k\otimes\iota_i} \ar[d]_{\Omega(\xi_i)}& 
k\otimes_RP_i\ar[r]^{k\otimes\psi_i}\ar[d]_{\Xi_i}&k\otimes_RX_i\ar[r]\ar[d]^{\xi_i}&0\\
0\ar[r]&\Omega(T)\ar[r]^{\iota_T} & P(T)\ar[r]^{\psi_T}&T\ar[r]&0
}
\end{equation}
of $\Lambda$-modules.
Then $\Xi_i$ and $\Omega(\xi_i)$
are $\Lambda$-module isomorphisms.
If there exists an $R\Lambda$-module isomorphism
$\nu:X_1\to X_2$ such that $\xi_2\circ(k\otimes \nu)=\xi_1$, 
then there exists an $R\Lambda$-module isomorphism
$\mu:Y_1\to Y_2$ with $\Omega(\xi_2)\circ(k\otimes \mu)=\Omega(\xi_1)$. In particular, this is true for any choice
of $\Omega(\xi_i)$, $i=1,2$, that makes the diagram in (\ref{eq:needthis!})
commutative.

\medskip

\noindent
\textit{Proof of Claim $2$}. 
We prove this for left modules.
Since $P_1$ is projective, there exist $R\Lambda$-module homomorphisms
$\tilde{\lambda}:P_1\to P_2$ and $\tilde{\mu}:Y_1\to Y_2$ making the following diagram of 
left $R\Lambda$-modules commutative
\begin{equation}
\label{omega1'}
\xymatrix{
0\ar[r]&Y_1\ar[r]^{\iota_1}\ar[d]_{\tilde{\mu}}& P_1\ar[r]^{\psi_1} \ar[d]_{\tilde{\lambda}}&X_1\ar[r]\ar[d]^{\nu}&0\\
0\ar[r]&Y_2\ar[r]^{\iota_2} & P_2\ar[r]^{\psi_2}&X_2\ar[r]&0.
}
\end{equation}
Since $\psi_2$ is an essential epimorphism by Claim 1, it follows that $\tilde{\lambda}$ is surjective, and hence
bijective because $P_1$ and $P_2$ are free $R$-modules of the same finite rank. 
This implies that $\tilde{\lambda}$ and $\tilde{\mu}$ are $R\Lambda$-module isomorphisms.

As in (\ref{eq:first}), we see that  for $i=1,2$, there exist $\Lambda$-module homomorphisms
$\Xi_i:k\otimes_RP_i\to P(T)$ and $\Omega(\xi_i):k\otimes_RY_i\to \Omega(T)$ such that we
obtain a commutative diagram as in (\ref{eq:needthis!}). On the other hand, if $\Xi_i$ and $\Omega(\xi_i)$ are
any $\Lambda$-module homomorphisms in a commutative diagram (\ref{eq:needthis!}), then
it follows, as in (\ref{eq:first}), that $\Xi_i$ and $\Omega(\xi_i)$ are $\Lambda$-module isomorphisms.

Since $\Lambda$ is injective as a left module over itself, 
$\Omega$ defines an auto\-equivalence of $\Lambda$-\underline{mod}. 
Because $\xi_2\circ (k\otimes \nu)
=\xi_1$ in $\Lambda$-mod, this implies that 
$$\Omega(\xi_2)\circ (k\otimes \tilde{\mu})=\Omega(\xi_1)$$
in $\Lambda$-\underline{mod}. 
This means that there exists a $\Lambda$-module homomorphism 
$p:k\otimes_RY_1\to k\otimes_RY_2$ such that $p$ factors through a
projective $\Lambda$-module and 
$$\Omega(\xi_2)\circ (k\otimes \tilde{\mu}+p)=\Omega(\xi_1)$$
in $\Lambda$-mod. 
Since finitely generated projective $\Lambda$-modules are injective, 
$p$ factors through $k\otimes\iota_1$, 
say $p=q\circ (k\otimes\iota_1)$ for some $\Lambda$-module homomorphism 
$q:k\otimes_RP_1\to k\otimes_RY_2$. 
Because $P_1$ is a projective $R\Lambda$-module, there exists
an $R\Lambda$-module homomorphism $q_R:P_1\to Y_2$ such that
$k\otimes q_R=q$. Hence we obtain a commutative diagram of $R\Lambda$-modules
\begin{equation}
\label{omega3}
\xymatrix{
0\ar[r]&Y_1\ar[r]^{\iota_1}\ar[d]_{\tilde{\mu}+q_R\circ\iota_1} & P_1
\ar[r]^{\psi_1}\ar[d]_{\tilde{\lambda}+\iota_2\circ q_R} &X_1\ar[r]\ar[d]^{\nu}&0\\
0\ar[r]&Y_2\ar[r]^{\iota_2} & P_2\ar[r]^{\psi_2}&X_2\ar[r]&0
}
\end{equation}
such that 
$$\Omega(\xi_2)\circ (k\otimes (\tilde{\mu}+q_R\circ\iota_1))=\Omega(\xi_1)\;.$$
Since $\psi_2$ is an essential epimorphism by Claim 1, we can argue as above to see
that $\mu=\tilde{\mu}+q_R\circ\iota_1$ is an $R\Lambda$-module isomorphism.
Because $\Omega(\xi_2)\circ (k\otimes \mu)=\Omega(\xi_1)$, this proves Claim 2.

\medskip

To prove Proposition \ref{prop:generalOmega}, we follow the strategy in 
\cite[Sect. 3.6]{velezthesis}. As we said at the beginning of the proof,
we may assume that $V$ has no projective direct summands, so that Claims
1 and 2 apply to both $T=V$ and $T=\Omega(V)^*$. Note that we can
use sequence (\ref{eq:syzygy2}) with $T=V$ in lieu of sequence 
(\ref{eq:syzygy1}) with $T=\Omega(V)^*$. In particular, we let
$P(\Omega(V)^*)=P(V)^*$ and $\psi_{\Omega(V)^*}=\iota_V^*$, which
implies $\Omega(\Omega(V)^*)=V^*$ and $\iota_{\Omega(V)^*}=\psi_V^*$.

Fix an Artinian ring $R$ in $\mathcal{C}$, and
let $\eta:P_R(V)\to P(V)$ be a projective $R\Lambda$-module cover of $P(V)$. 
In other words, $P_R(V)$ is a projective $R\Lambda$-module and $\eta$ is an essential epimorphism. Equivalently,
there exists a $\Lambda$-module isomorphism 
$\bar{\eta}:k\otimes_RP_R(V)\to P(V)$ and
$\eta=\bar{\eta}\circ p_{P_R(V)}$ where $p_{P_R(V)}:P_R(V)\to k\otimes_R P_R(V)$ 
sends $x\in P_R(V)$ to $1\otimes x$.

If $(M,\phi)$ is a lift of $V$ over $R$, then
there exists an $R\Lambda$-module homomorphism $\psi_M:P_R(V)\to M$ such that
$\psi_M:P_R(V)\to M$ is a projective $R\Lambda$-module cover of $M$. Define $\Omega_R(M)$
to be the kernel of $\psi_M$. 
As in (\ref{eq:needthis!}) in Claim 2, we have a commutative diagram of $\Lambda$-modules
\begin{equation}
\label{eq:omega1}
\xymatrix{
0\ar[r]&k\otimes_R\Omega_R(M)\ar[r]\ar[d]_{\Omega(\phi)} & k\otimes_R P_R(V)
\ar[r]^{k\otimes\psi_M}\ar[d]_{\Phi_M} &k\otimes_RM\ar[r]\ar[d]^{\phi}&0\\
0\ar[r]&\Omega(V)\ar[r]^{\iota_V} & P(V)\ar[r]^{\psi_V}&V\ar[r]&0
}
\end{equation}
where $\Phi_M$ and $\Omega(\phi)$ are $\Lambda$-module isomorphisms.
Hence $(\Omega_R(M),\Omega(\phi))$ is a lift of $\Omega(V)$ over $R$. 
By Claim 2, we obtain a well-defined map
\begin{equation}
\label{eq:omega5}
\xymatrix @R=.2pc{
g_{\Omega,R}:&\mathrm{Def}_{\Lambda}(V,R)\ar[r]&\mathrm{Def}_{\Lambda}(\Omega(V),R)\\
&[M,\phi]\ar@{|->}[r]& [\Omega_R(M),\Omega(\phi)]\;.}
\end{equation}
Let $\alpha:R\to R'$ be a morphism in $\mathcal{C}$, and
consider the induced lift $(R'\otimes_{R,\alpha}M,(\phi)_\alpha)$ of $V$ over $R'$.
Then
it follows from Claim 2 that 
we have an isomorphism
$$(\Omega_{R'}(R'\otimes_RM),\Omega((\phi)_\alpha))\cong
(R'\otimes_{R,\alpha}\Omega_R(M),(\Omega(\phi))_\alpha)$$
as lifts of $\Omega(V)$ over $R'$. This proves that $g_{\Omega,R}$ is natural with respect 
to morphisms $\alpha:R\to R'$ in $\mathcal{C}$. 

We next prove that $g_{\Omega,R}$ in (\ref{eq:omega5}) is surjective. Let $(U,\rho)$ be a lift
of $\Omega(V)$ over $R$. Since $\rho:k\otimes_RU\to \Omega(V)$ is a $\Lambda$-module
isomorphism and since $k\otimes_RP_R(V)\cong P(V)$, 
it follows from Remark \ref{rem:josepaper}(i)
that there exists an $R\Lambda$-module homomorphism 
$\varphi:U\to P_R(V)$ such that we have a commutative diagram of $R\Lambda$-modules
\begin{equation}
\label{eq:omega9}
\xymatrix{
0\ar[r]&U\ar[r]^{\varphi}\ar[d]_{\rho\circ p_U} & P_R(V)
\ar[r]^{\pi}\ar[d]_{\eta} &\mathrm{Coker}(\varphi)\ar[r]\ar[d]^{\zeta}&0\\
0\ar[r]&\Omega(V)\ar[r]^{\iota_V} & P(V)\ar[r]^{\psi_V}&V\ar[r]&0
}
\end{equation}
where $p_U:U\to k\otimes_R U$ sends $u\in U$ to $1\otimes u$,
$\pi$ is the natural projection and $\zeta$ is the $R\Lambda$-module homomorphism induced 
by $\eta$. Note that since $U$ and $P_R(V)$ are free $R$-modules of finite rank and 
since $\iota_V$
is injective, it follows by Nakayama's Lemma that $\varphi$ is also injective.
Moreover, by lifting bases from $k$ to $R$,
we see that
$\mathrm{Coker}(\varphi)$ is a free $R$-module. 
Tensoring the top row of (\ref{eq:omega9}) with $k$ over $R$, we see that $\zeta$ induces a
$\Lambda$-module isomorphism $\bar{\zeta}: k\otimes_R\mathrm{Coker}(\varphi) \to V$.
Hence $(\mathrm{Coker}(\varphi),\bar{\zeta})$ is a lift of $V$ over $R$.
By Claim 1, $\pi:P_R(V)\to \mathrm{Coker}(\varphi)$ is an essential 
epimorphism.
This implies
that $g_{\Omega,R}([\mathrm{Coker}(\varphi),\bar{\zeta}]) = [U,\rho]$, proving that $g_{\Omega,R}$ is 
surjective.

Finally, we show that $g_{\Omega,R}$ in (\ref{eq:omega5}) is injective. Let $(L_1,\phi_1)$, $(L_2,\phi_2)$ be lifts of $V$ over
$R$ such that 
$(\Omega_R(L_1),\Omega(\phi_1))$ and $(\Omega_R(L_2),\Omega(\phi_2))$ are isomorphic
as lifts of $\Omega(V)$ over $R$. Let $f:\Omega_R(L_1)\to \Omega_R(L_2)$ be an $R\Lambda$-module
isomorphism such that 
$$\Omega(\phi_2)\circ (k\otimes f)=\Omega(\phi_1)\,.$$ 
Let $i=1,2$. Note that $(L_i^*,(\phi_i^*)^{-1})$ is a lift of 
$V^*$ over $R$, and $(\Omega_R(L_i)^*,(\Omega(\phi_i)^*)^{-1})$
is a lift of $\Omega(V)^*$ over $R$. 
We have a short exact sequence of right $R\Lambda$-modules
$$0\to L_i^*\xrightarrow{\psi_{L_i}^*} P_R(V)^*
\xrightarrow{\iota_{L_i}^*} \Omega_R(L_i)^*\to 0$$
where $P_R(V)^*$ is a projective right $R\Lambda$-module, by Remark
\ref{rem:josepaper}(ii), and $k\otimes_R P_R(V)^*\cong P(V)^*$.
Moreover, $(f^*)^{-1}: \Omega_R(L_1)^* \to \Omega_R(L_2)^*$
is an isomorphism of right $R\Lambda$-modules satisfying
$$(\Omega(\phi_2)^*)^{-1}\circ (k\otimes (f^*)^{-1}) =(\Omega(\phi_1)^*)^{-1}\;.$$ 
Therefore, 
it follows from Claim 2
that there exists an isomorphism 
$\tilde{h}:L_2^*\to L_1^*$ of right $R\Lambda$-modules such that 
\begin{equation}
\label{eq:neceq}
(\phi_2^*)^{-1}\circ (k\otimes \tilde{h}^{-1}) = (\phi_1^*)^{-1}\,.
\end{equation}
In other words, $(L_1^*,(\phi_1^*)^{-1})$ and $(L_2^*,(\phi_2^*)^{-1})$
are isomorphic lifts of $V^*$ over $R$.

For each finitely generated left $R\Lambda$-module $M$ that is free as an $R$-module, 
let $\delta_M:M\to M^{**}$ be the $R\Lambda$-module isomorphism
given by evaluation. Note that
if $N$ is another finitely generated $R\Lambda$-module that is free as an $R$-module and
$\lambda:M\to N$ is an $R\Lambda$-module homomorphism, then
$\lambda^{**}=\delta_N\circ \lambda\circ (\delta_M)^{-1}$.
Similarly, we can define a $\Lambda$-module isomorphism 
$\delta_V:V\to V^{**}$. Define
$$h=(\delta_{L_2})^{-1}\circ \tilde{h}^*\circ \delta_{L_1}:\quad L_1\to L_2\;.$$
Using (\ref{eq:neceq}), it follows that $h$ is an $R\Lambda$-module isomorphism such that $\phi_2\circ (k\otimes h)=\phi_1$.
Hence $(L_1,\phi_1)$ and $(L_2,\phi_2)$ are isomorphic lifts of $V$ over $R$, proving 
that $g_{\Omega,R}$ is injective. 

It follows that the syzygy functor $\Omega$ induces a natural isomorphism
between the restrictions of the deformation functors $\mathrm{Def}_{\Lambda}(V,-)$ and 
$\mathrm{Def}_{\Lambda}(\Omega(V),-)$ to $\mathcal{C}$.
Since the deformation functors $\mathrm{Def}_{\Lambda}(V,-)$ and 
$\mathrm{Def}_{\Lambda}(\Omega(V),-)$ are continuous by \cite[Prop. 2.1]{blehervelez}, 
this proves Proposition \ref{prop:generalOmega}.
\end{proof}

We next give a necessary and sufficient criterion for a versal
deformation ring to be universal. To state the result, we need the following definition
(see, for example, \cite[Def. 2.5]{bcillinois}).

\begin{dfn}
\label{def:centralizerlifting}
Let $\Lambda$ be an arbitrary finite dimensional $k$-algebra and let $V$ be an arbitrary finitely generated
$\Lambda$-module. 
As at the beginning of Section \ref{s:prelim}, choose a $k$-basis of $V$, and let
$\rho_V:\Lambda\to \mathrm{Mat}_n(k)$ be the $k$-algebra homomorphism giving the
action of $\Lambda$ on $V$ with respect to this basis.

Let $\alpha:A_1\to A_0$
be a small extension in $\mathcal{C}$, which means that $\alpha$ is a 
surjective morphism in $\mathcal{C}$ and its kernel is a non-zero principal
ideal $(t)$ of $A_1$ such that $\mathfrak{m}_{A_1}\cdot t=0$. For $j\in\{0,1\}$ define
$$G_{A_j}=\mathrm{Ker}\left(\mathrm{GL}_n(A_j)\xrightarrow{\pi_{A_j}} \mathrm{GL}_n(k)\right)\;.$$
Note that $\alpha$ induces a surjective homomorphism $G_{A_1}\to G_{A_0}$.
Suppose $\tau_1:\Lambda\to\mathrm{Mat}_n(A_1)$ is a lift of $\rho_V$ over $A_1$, and define
$\tau_0=\alpha\circ\tau_1$. For $j\in\{0,1\}$ define
$$Z(\tau_j)=\{T_i\in G_{A_j}\;|\;T_j\,\tau_j=\tau_j\,T_j\}\;.$$
We say the deformation functor $\mathrm{Def}_{\Lambda}(V,-)=\mathrm{Def}_{\Lambda}(\rho_V,-)$
has the \emph{centralizer lifting property} if for all small extensions $\alpha:A_1\to A_0$
in $\mathcal{C}$ and for all lifts $\tau_1,\tau_0$ as above, the natural homomorphism
$$Z(\tau_1)\to Z(\tau_0)$$
induced by $\alpha$
is surjective. Note that the surjectivity of this map only depends on the strict equivalence
class $[\tau_1]$ and the ring homomorphism $\alpha$ but not on the choice of representative
$\tau_1$ in $[\tau_1]$.
\end{dfn}

Using Schlessinger's criterion $(\HH_4)$ (see \cite[Sect. 2]{Sch}) together with the fact that 
$\mathrm{Def}_{\Lambda}(V,-)$ is continuous by \cite[Prop. 2.1]{blehervelez}, 
the following result is proved similarly to 
\cite[Lemma 1]{maz} (see also \cite[Thm. 2.7(ii)]{bcillinois}).

\begin{lemma}
\label{lem:centralizerlifting}
Let $\Lambda$ and $V$ be as in Definition $\ref{def:centralizerlifting}$.
The versal deformation ring $R(\Lambda,V)$ is universal if and only if the
deformation functor $\mathrm{Def}_{\Lambda}(V,-)$ has the centralizer lifting property.
\end{lemma}

%%%%%%%%%%%%%%%%%%%%%%%%%%%%%%%%%%%%%%%%%%%%%%%%%%%%%%%%%
%% Self-injective split basic Nakayama algebras
%%%%%%%%%%%%%%%%%%%%%%%%%%%%%%%%%%%%%%%%%%%%%%%%%%%%%%%%%

\section{Self-injective split basic Nakayama algebras}
\label{s:nakayama}
\setcounter{equation}{0}

The goal of this section is to prove Theorem \ref{thm:supermain}.
Let $k$ be an arbitrary field. Recall that a finite dimensional $k$-algebra $\Lambda$ is called
a \emph{Nakayama algebra} if both the indecomposable
projective and the indecomposable injective $\Lambda$-modules are uniserial.
If $\mathrm{rad}(\Lambda)$ denotes the 
Jacobson radical of $\Lambda$, then $\Lambda$ is said to be \emph{split basic} over $k$ if 
$\Lambda/\mathrm{rad}(\Lambda)$ is isomorphic to a direct product of copies of $k$.

Let $e\ge 1$ and $\ell \ge 2$ be integers, and let $Q_e$, $\mathcal{J}$ and $\mathcal{N}(e,\ell)$ be as 
in Definition \ref{def:needthis}(a). In other words, $Q_e$ is the circular quiver with $e$ 
vertices, labeled $1,\ldots,e$, and $e$ arrows,
labeled $\alpha_1,\ldots,\alpha_e$, such that $\alpha_v: v\to v+1$, for $1\le v\le e-1$, and 
$\alpha_e:e\to 1$:
\begin{center}
\begin{tikzpicture}

\def \n {5}
\def \radius {1}
\def \margin {15} % margin in angles, depends on the radius

\foreach \s in {1,2,\n}
{
  %\node at ({360/\n * (\s - 1)}:\radius) {$\s$};
  \draw[->, black,thick] ({360/\n * (\s - 1)+\margin}:\radius) 
    arc ({360/\n * (\s - 1)+\margin}:{360/\n * (\s)-\margin}:\radius);
}

\node at (-2,0) {$Q_e=$};
\node at (1,0) {1};
\node at (.309,.951) {2};
\node at (-.809,.588) {3};
\node at (.309,-.951) {$e$};
\node at (1.2,.6) {$\alpha_1$};
\node at (-.3,1.2) {$\alpha_2$};
\node at (1.1,-.7) {$\alpha_e$};

\draw[->,thick, dashed] (-.940,.342) arc(160:274:1);

\end{tikzpicture} 
%}
\end{center}
Moreover, $\mathcal{J}$ is the ideal of $k\,Q_e$ generated by all arrows and
$$\mathcal{N}(e,\ell) =k\,Q_e/\mathcal{J}^\ell\;.$$
For $1\le j\le e$, let $S_j$ be the simple $\mathcal{N}(e,\ell)$-module corresponding to
the vertex $j$. Write 
$$\ell=\mu\, e + \ell'$$ 
as in (\ref{eq:ell}),
where $\mu,\ell'\ge 0$ are integers and $\ell'\le e-1$.
Since we assume $\ell\ge 2$, it follows in the case when $e=1$ that $\mu\ge 2$.

The algebra $\mathcal{N}(e,\ell)$ is an indecomposable split basic Nakayama algebra over $k$.
The projective indecomposable $\mathcal{N}(e,\ell)$-modules $P_1,\ldots,P_e$ and the
injective indecomposable $\mathcal{N}(e,\ell)$-modules $E_1,\ldots,E_e$
are uniserial of length $\ell$ such that, for $1\le j\le e$, 
$P_j/\mathrm{rad}(P_j)\cong S_j$ and $\mathrm{soc}(E_j)\cong S_j$.
In other words, the descending composition factors of $P_j$, for $1\le j\le e$,
are given by the sequence of $\ell$ simple $\mathcal{N}(e,\ell)$-modules
$$S_j,S_{j+1},\ldots,S_e,S_1,\ldots,S_{j-1},S_j,\ldots,S_{j-1},S_j,\ldots,\ldots,S_{j-1},S_j,\ldots,S_{j-1+\ell'}$$
where $S_j$ occurs $\mu$ (resp. $\mu+1$) times as a composition factor
when $\ell'=0$ (resp. $\ell'\ge 1$). Note that
$$P_j\cong E_{j-1+\ell'}$$
where we take indices modulo $e$. In particular, 
$\mathcal{N}(e,\ell)$ is a Frobenius algebra for all $e\ge 1$ and all $\ell \ge 2$,
and $\mathcal{N}(e,\ell)$ is a symmetric algebra if and only if $\ell'=1$. 

It is well-known (see, for example, \cite[p. 243]{GabrielRiedtmann})
that every indecomposable self-injective non-semisimple split basic Nakayama algebra over $k$ 
is isomorphic to $\mathcal{N}(e,\ell)$ for appropriate integers $e\ge 1$ and $\ell\ge 2$. 

For the remainder of this section, fix integers $e\ge 1$ and $\ell\ge 2$, and define
\begin{equation}
\label{eq:Nsimplify}
\mathcal{N}=\mathcal{N}(e,\ell)\;.
\end{equation}
There are 
precisely $e\cdot \ell$ isomorphism classes of indecomposable 
$\mathcal{N}$-modules. A representative of each such isomorphism class is uniquely determined by its
top radical quotient, which we will call its top, and its length. 
In the following, we will concentrate on indecomposable $\mathcal{N}$-modules 
whose top is isomorphic to $S_1$. 

\begin{dfn}
\label{def:notation}
Let $0\le n\le \mu$ and $0\le i\le e-1$ be integers, and define 
$\ell_{n,i}=n\,e+i$. Assume $\ell_{n,i}\le \ell$. If $\ell_{n,i}=0$, define $V_{0,0}=0$. Now suppose $\ell_{n,i}\ge 1$.

Define $V_{n,i}$ to be an indecomposable $\mathcal{N}$-module with
$\mathrm{top}(V_{n,i})=S_1$ and $\mathrm{dim}_k\,V_{n,i}=\ell_{n,i}$. Then $V_{n,i}$ is
unique up to isomorphism. 
The descending composition factors of $V_{n,i}$ are given by the sequence of $\ell_{n,i}$ simple 
$\mathcal{N}$-modules
$$S_1,S_2,\ldots,S_e,S_1,\ldots,S_e,S_1,\ldots,\ldots,S_e,S_1,\ldots,S_{i}$$
where each of $S_{i+1},\ldots,S_e$ occurs $n$ times and, if $i\ge 1$, each of $S_1,\ldots,S_i$ occurs $n+1$ times.
Define
\begin{equation}
\label{eq:delta}
\theta({v,n,i})=\left\{\begin{array}{ccl}n+1&:&1\le v\le i,\\
n&:&i+1\le v \le e.\end{array}\right.
\end{equation}
We fix a representation of $V_{n,i}$
$$\rho_{n,i}:\qquad\mathcal{N}\longrightarrow \mathrm{Mat}_{\ell_{n,i}}(k)$$
as follows. 

If $n=0$ then $\ell_{n,i}=i<e$. In this case, $\rho_{n,i}(v)$ (resp. $\rho_{n,i}(\alpha_v)$) is the zero matrix for $i+1\le v\le e$.
Moreover, for $1\le v\le i$,
\begin{equation}
\label{eq:rho00}
\rho_{n,i}(v) = \left(\begin{array}{cccc}\delta_{v,1}&0&\cdots&0\\ 0&\delta_{v,2}&\;\ddots&\vdots\\
\vdots&\ddots&\ddots&0\\ 0&\cdots&0&\delta_{v,i}\end{array}\right),\qquad
\rho_{n,i}(\alpha_v) = \left(\begin{array}{ccccc}0&\cdots&\cdots&\cdots&0\\
\delta_{v,1}&0&&&\vdots\\ 0&\delta_{v,2}&\;\ddots&&\vdots\\
\vdots&\ddots&\ddots&\;\ddots&\vdots\\ 0&\cdots&0&\delta_{v,i-1}&0\end{array}\right)
\end{equation}
where $\delta_{v,j}$ is the Kronecker delta.

If $n\ge 1$
then $\ell_{n,i}\ge e$. In this case,
for $1\le v\le e$, $\rho_{n,i}(v)$ and $\rho_{n,i}(\alpha_v)$ are $e\times e$ block matrices
\begin{eqnarray}
\label{eq:rho11vertices}
\rho_{n,i}(v) &=& \left(\begin{array}{cccc}\delta_{v,1}\,\mathrm{I}_{\theta({1,n,i})}&0&\quad\cdots\quad&0\\ 
0&\delta_{v,2}\,\mathrm{I}_{\theta({2,n,i})}&\;\ddots&\vdots\\
\vdots&\ddots&\ddots&0\\ 0&\cdots&0&\delta_{v,e}\,\mathrm{I}_{\theta({e,n,i})}\end{array}\right),\\[3ex]
\label{eq:rho11}
\rho_{n,i}(\alpha_v) &=& \left(\begin{array}{ccccc}0&0&\quad\cdots\quad&0&\delta_{v,e}A_{e,n,i}\\[1ex]
\delta_{v,1}A_{1,n,i}&0&&&0\\[1ex] 0&\delta_{v,2}A_{2,n,i}&\;\ddots&&\vdots\\[2ex]
\vdots&\ddots&\;\ddots&0&0\\[1ex] 0&\cdots&0&\delta_{v,e-1}A_{e-1,n,i}&0\end{array}\right)
\end{eqnarray}
where 
$A_{v,n,i}$ is a $\theta({v+1,n,i})\times \theta({v,n,i})$ matrix for $1\le v\le e-1$ and
$A_{e,n,i}$ is a $\theta({1,n,i})\times \theta({e,n,i})$ matrix. Moreover,
\begin{equation}
\label{eq:Avni}
A_{v,n,i}=\left\{\begin{array}{ccl}
\mathrm{Id}_{\theta({v+1,n,i})\times \theta({v,n,i})}^{\mathrm{max}}&:&1\le v\le e-1\;,\\[2ex]
\mathrm{Id}_{\theta({1,n,i})\times\theta({e,n,i})}^{\theta({1,n,i})-1}&:&v=e\;,
\end{array}\right.
\end{equation}
where $\mathrm{Id}_{x\times y}^{r}$ denotes the $x\times y$ matrix of rank $r$ of the form
\begin{equation}
\label{eq:Id}
\mathrm{Id}_{x\times y}^r=\left(\begin{array}{lc} {0} & {0}\\ \mathrm{I}_{r} &{0}\end{array}\right)
\qquad\mbox{if $1\le r\le \mathrm{min}(x,y)$,}
\end{equation}
$\mathrm{Id}_{x\times y}^{\mathrm{max}}=\mathrm{Id}_{x\times y}^{\mathrm{min}(x,y)}$ and
$\mathrm{Id}_{x\times y}^0$ is the $x\times y$ zero matrix.

For $n\ge 0$,
we denote the $k$-basis of $V_{n,i}$ with respect to which we obtain the matrix representation $\rho_{n,i}$ 
by
\begin{equation}
\label{eq:basis}
\mathcal{B}_{n,i}=\{\vec{b}_{n,i,v,w}\;|\; 1\le v\le e, \;1\le w\le \theta(v,n,i)\}\;,
\end{equation}
where $\vec{b}_{n,i,v,w}$ is the column vector of length $\ell_{n,i}$ whose entry at the coordinate $w+\sum_{u=1}^{v-1}\theta(u,n,i)$  is 1 and
whose all other entries are 0. Note that
$\mathcal{B}_{0,i}=\{\vec{b}_{n,i,v,1}\;|\; 1\le v\le i\}$.
\end{dfn}

\begin{lemma}
\label{lem:ext1}
Let $0\le n\le \mu$ and $0\le i\le e-1$ be integers such that $\ell_{n,i}=n\,e+i$ satisfies $1\le \ell_{n,i}\le \lfloor \ell/2 \rfloor$.
Let $V_{n,i}$ be as in Definition $\ref{def:notation}$.
Then $\mathrm{Ext}^1_{\mathcal{N}}(V_{n,i},V_{n,i})\cong \mathrm{Hom}_\mathcal{N}(\Omega(V_{n,i}),V_{n,i})$ and
$\mathrm{dim}_k\mathrm{Ext}^1_{\mathcal{N}}(V_{n,i},V_{n,i})=n$.
\end{lemma}

\begin{proof}
Since $\mathcal{N}$ is self-injective, it follows that
$$\mathrm{Ext}^1_{\mathcal{N}}(V_{n,i},V_{n,i})\cong\underline{\mathrm{Hom}}_\mathcal{N}(\Omega(V_{n,i}),V_{n,i})$$
as $k$-vector spaces.
Note that $\Omega(V_{n,i})$ is an indecomposable $\mathcal{N}$-module 
with $\mathrm{dim}_k\,\Omega(V_{n,i})=\ell-\ell_{n,i}$ and $\mathrm{top}(\Omega(V_{n,i}))=S_{i+1}$.
Since $\mathrm{dim}_k\,\Omega(V_{n,i})\ge \mathrm{dim}_k\,V_{n,i}$ and since both $V_{n,i}$ and $\Omega(V_{n,i})$ 
are uniserial, it follows that $\mathrm{dim}_k\,\mathrm{Hom}_\mathcal{N}(\Omega(V_{n,i}),V_{n,i})$
equals the multiplicity of $S_{i+1}$ as a composition factor of $V_{n,i}$. Since this number is equal to $n$
and since none of the $\mathcal{N}$-module homomorphisms from $\Omega(V_{n,i})$ to $V_{n,i}$ factors through a 
projective $\mathcal{N}$-module, the result follows.
\end{proof}

For $n\ge 1$ (which implies $\mu\ge 1$) and $\ell_{n,i}\le \lfloor \ell/2 \rfloor$, we need an explicit description of a $k$-basis of
$\mathrm{Ext}^1_{\mathcal{N}}(V_{n,i},V_{n,i})$ in terms of short exact sequences. We use the following definitions.

\begin{dfn}
\label{def:maps}
Suppose $M=M(a,\theta_M)$ is an indecomposable $\mathcal{N}$-module with $\mathrm{top}(M)=S_a$ such that
the multiplicity of $S_a$ as a composition factor of $M$ is $\theta_M$.
Fix an element $z_M\in M$ such that $z_M\not\in \mathrm{rad}(M)$ and
$\alpha_v\, z_M=0$ for $v\in\{1,\ldots,e\}-\{a\}$. We call $z_M$ a top element of $M$.
Since $M$ is uniserial, every $\mathcal{N}$-module homomorphism $\beta$ with domain $M$ is uniquely determined by
$\beta(z_M)$. 

Suppose  $0\le n\le \mu$ and $0\le i\le e-1$ are integers such that $\ell_{n,i}=n\,e+i\le \ell$.
Let $V_{n,i}$ be as in Definition \ref{def:notation}. Define $\theta(a,0,0)=0$. 
If $\ell_{n,i}\ge 1$, then $\theta(a,n,i)$ from (\ref{eq:delta}) is the multiplicity of
$S_a$ as a composition factor of $V_{n,i}$. For $0\le j\le \mathrm{min}(\theta(a,n,i),\theta_M)$, define
$$\beta({M,V_{n,i},j}):\quad M \to V_{n,i}$$
to be the $\mathcal{N}$-module homomorphism such that $\beta({M,V_{n,i},0})$ sends $z_M$ to $0$ and
$\beta({M,V_{n,i},j})$ sends $z_M$ to $\vec{b}_{n,i,a,\theta(a,n,i)-j+1}$ for $1\le j\le \mathrm{min}(\theta(a,n,i),\theta_M)$.

For $0\le c,d\le\mu$, define
$$\beta_{c,d}:\qquad V_{c,i}\to V_{d,i}$$
to be $\beta_{c,d} = \beta({V_{c,i},V_{d,i},\mathrm{min}(\theta(1,c,i),\theta(1,d,i))})$ where $z_{V_{c,i}}=\vec{b}_{c,i,1,1}$.
In particular, if $c=d$ then $\beta_{c,d}$ is the identity morphism, if
$c> d$ then $\beta_{c,d}$ is the natural projection from $V_{c,i}$ onto $V_{d,i}$, and
if $c<d$ then $\beta_{c,d}$ is the natural inclusion of $V_{c,i}$ into $V_{d,i}$.
\end{dfn}

\begin{dfn}
\label{def:extsequences}
Let $1\le n\le \mu$ and $0\le i\le e-1$ be integers such that $\ell_{n,i}=n\,e+i$ satisfies $\ell_{n,i}\le \lfloor \ell/2 \rfloor$.
For $s\in\{1,\ldots,n\}$, define a short exact sequence of $\mathcal{N}$-modules
$$\mathcal{E}_s:\qquad0 \rightarrow V_{n,i} \xrightarrow{\;\iota_s\;} 
V_{n+s,i} \oplus V_{n-s,i} \xrightarrow{\;\pi_s\;} V_{n,i} \rightarrow 0$$
where 
$$\iota_s=\left(\begin{array}{r}\beta_{n,n+s}\\-\beta_{n,n-s}\end{array}\right)\;, \qquad
\pi_s=\left(\begin{array}{cc}\beta_{n+s,n}&\beta_{n-s,n}\end{array}\right)$$
and
$\beta_{n,n+s}$, $\beta_{n,n-s}$, $\beta_{n+s,n}$, $\beta_{n-s,n}$ are as in Definition \ref{def:maps}.
\end{dfn}

\begin{lemma}
\label{lem:ext2}
Let $1\le n\le \mu$ and $0\le i\le e-1$ be integers such that $\ell_{n,i}=n\,e+i$ satisfies $\ell_{n,i}\le \lfloor \ell/2 \rfloor$.
The short exact sequences $\mathcal{E}_1,\ldots,\mathcal{E}_n$ from Definition
$\ref{def:extsequences}$ define $k$-linearly independent elements, and hence a 
$k$-basis, of $\mathrm{Ext}_\mathcal{N}^1(V_{n,i},V_{n,i})$.
\end{lemma}

\begin{proof}
As in (\ref{eq:ell}), write $\ell=\mu\, e + \ell'$.
It follows from the assumptions that $\Omega(V_{n,i})$ is an indecomposable $\mathcal{N}$-module 
whose top is isomorphic to $S_{i+1}$ and $\mathrm{dim}_k\,\Omega(V_{n,i})=\ell-\ell_{n,i}\ge e$.
Let $\theta_{\Omega,n,i}$ be the multiplicity of $S_{i+1}$ as a composition factor of $\Omega(V_{n,i})$. Note
that $\theta_{\Omega,n,i}=\mu-n$ if $i+1> \ell'$ and $\theta_{\Omega,n,i}=\mu-n+1$ if $i+1\le \ell'$.
Fix $s\in\{1,2,\ldots,n\}$. We have the following commutative diagram of $\mathcal{N}$-modules with exact rows
\begin{equation}
\label{eq:seq!}
\xymatrix{
0\ar[r]&  \Omega(V_{n,i}) \ar[rr]^{\iota} \ar[dd]_{\omega_s} && P_1=V_{\mu,\ell'} \ar[rr]^{\pi} \ar[dd]^{\small \left(\begin{array}{c}\kappa_s\\0\end{array}\right)} && V_{n,i}\ar[r]\ar@{=}[dd] & 0\\
&&&&&&&\\
0 \ar[r]& V_{n,i} \ar[rr]^(.4){\iota_s} &&
V_{n+s,i} \oplus V_{n-s,i} \ar[rr]^{\pi_s}&& V_{n,i} \ar[r]& 0}
\end{equation}
where 
$$\begin{array}{rclrcl}
\iota &=& \beta(\Omega(V_{n,i}),V_{\mu,\ell'}, \theta_{\Omega,n,i})\;,&
\pi&=&\beta(V_{\mu,\ell'}, V_{n,i},\theta(1,n,i))\;,\\
\omega_s&=&\beta(\Omega(V_{n,i}),V_{n,i}, s)\;,&
\kappa_s&=&\beta(V_{\mu,\ell'}, V_{n+s,i},\theta(1,n+s,i))\;.
\end{array}$$
By Lemma \ref{lem:ext1}, it follows that the short exact sequence $\mathcal{E}_s$, which is the bottom row of (\ref{eq:seq!}), 
corresponds to the map
$\omega_s\in\mathrm{Hom}_\mathcal{N}(\Omega(V_{n,i}),V_{n,i})$. Therefore, to prove Lemma \ref{lem:ext2}, it suffices to show that $\omega_1,\ldots,\omega_n$
are $k$-linearly independent as elements of $\mathrm{Hom}_\mathcal{N}(\Omega(V_{n,i}),V_{n,i})$. Considering the images of $\omega_1,\ldots,\omega_n$
in $V_{n,i}$, we see that
$$\mathrm{Im}(\omega_1)\subset \mathrm{Im}(\omega_2)\subset \cdots \subset \mathrm{Im}(\omega_n)$$
where all inclusions are proper. This implies right away that $\omega_1,\ldots,\omega_n$
are $k$-linearly independent in $\mathrm{Hom}_\mathcal{N}(\Omega(V_{n,i}),V_{n,i})$, completing the proof of Lemma \ref{lem:ext2}.
\end{proof}

We next use the short exact sequences $\mathcal{E}_1,\ldots,\mathcal{E}_n$ from Definition \ref{def:extsequences} to define $n$
$k$-linearly independent deformations of $V_{n,i}$ over the ring of dual numbers $k[\epsilon]$.

\begin{dfn}
\label{def:dualnumberlifts}
Fix integers $1\le n\le \mu$ and $0\le i\le e-1$ such that $\ell_{n,i}=n\,e+i$ satisfies $\ell_{n,i}\le \lfloor \ell/2 \rfloor$.
Fix $s\in\{1,\ldots,n\}$, and let $\mathcal{E}_s$ be the short exact sequence from Definition \ref{def:extsequences}.
Define 
$$M_s=V_{n+s,i} \oplus V_{n-s,i}\;,$$
i.e. $M_s$ is the center module of the sequence $\mathcal{E}_s$. Also define an $\mathcal{N}$-module endomorphism $\epsilon_s: M_s\to M_s$ by
$$\epsilon_s=\iota_s\circ\pi_s\;.$$
Then $\epsilon_s\circ\epsilon_s=0$.
Define $R_s=k[t_s]/(t_s^2)$, so $R_s\cong k[\epsilon_s]$ which is isomorphic to the ring of dual numbers
$k[\epsilon]$. Then $M_s$ is a free $R_s$-module of rank $\ell_{n,i}=n\,e + i$, where we let $t_s$
act as the endomorphism $\epsilon_s$.
More precisely, if we view $V_{n+s,i}$ as a $k$-subspace of $M_s$ and use the $k$-basis
$\mathcal{B}_{n+s,i}$ of $V_{n+s,i}$ from (\ref{eq:basis}), then an $R_s$-basis of $M_s$ is given by
$$\{\vec{b}_{n+s,i,v,w}\;|\;1\le v \le e, \; 1\le w \le \theta(v,n,i)\}\;,$$
where $\theta(v,n,i)$ is as in (\ref{eq:delta}).
With respect to this $R_s$-basis, we obtain the following representation 
$$\rho_{n,i,s}:\qquad \mathcal{N} \to \mathrm{Mat}_{\ell_{n,i}}(R_s)$$
of $M_s$.
Viewing $k$ as a $k$-subalgebra of $R_s$ and the notation from Definition \ref{def:notation}, we have for all $1\le v\le e$:
\begin{eqnarray*}
\rho_{n,i,s}(v)&=&\rho_{n,i}(v) \,,\\[2ex]
\rho_{n,i,s}(\alpha_v) &=& \left\{\begin{array}{ccl}
\rho_{n,i}(\alpha_v)&:&v\not\equiv i\mod e\;,\\[2ex]
\rho_{n,i}(\alpha_v)+T_s&:&v\equiv  i\mod e\;,
\end{array}\right.
\end{eqnarray*}
where $T_s$ is an $e\times e$ block matrix
$$T_s=\left(T_{s,a,b}\right)_{1\le a,b\le e}$$
such that 
$T_{s,a,b}$ is a $\theta(a,n,i)\times \theta(b,n,i)$ matrix.
Moreover, $T_{s,a,b}$ is the zero matrix unless $a=i+1$ and $b\equiv i\mod e$, and
$$T_{s,a,b} =\left(\begin{array}{ccc|c}0&\cdots&0&\\ \vdots&&\vdots&\vec{t}_s\\ 0&\cdots&0&\end{array}\right)
\quad \mbox{for $a=i+1$ and $b\equiv i\mod e$}\;,$$
where $\vec{t}_s$ is the column vector of length $n=\theta(i+1,n,i)$ whose $(n-s+1)$-th entry is $t_s$ and whose all other entries are zero.
Since $k\otimes_{R_s} M_s\cong M_s/t_s\,M_s\cong M_s/\mathrm{Im}(\iota_s) = M_s/\mathrm{Ker}(\pi_s)$, we can define
$$\phi_s : k\otimes_{R_s} M_s\to V_{n,i}$$ 
to be the isomorphism induced by $\pi_s$. Hence we obtain a lift $(M_s,\phi_s)$ of $V_{n,i}$ over $R_s\cong k[\epsilon]$
corresponding to the sequence $\mathcal{E}_s$. Because the reduction map $\pi_{R_s}:R_s\to k$ is the $k$-algebra homomorphism given by sending
$t_s$ to $0$, the deformation $[M_s,\phi_s]$ of $V_{n,i}$ over $R_s\cong k[\epsilon]$ can be identified with the strict equivalence
class $[\rho_{n,i,s}]$. Since the tangent space of the deformation functor $\mathrm{Def}_\mathcal{N}(V_{n,i},-)$
is isomorphic to $\mathrm{Ext}^1_\mathcal{N}(V_{n,i},V_{n,i}$), it follows from Lemma \ref{lem:ext2} that the set of deformations
$$\{[M_s,\phi_s]\;|\;1\le s\le n\}= \{[\rho_{n,i,s}]\;|\;1\le s\le n\}$$
provides a $k$-basis of $\mathrm{Def}_\mathcal{N}(V_{n,i},k[\epsilon])$. 
\end{dfn}

Let $Q_{e,0}=\{1,\ldots,e\}$ (resp. $Q_{e,1}=\{\alpha_1,\ldots,\alpha_e\}$) be the set of vertices (resp.
arrows) in the circular quiver $Q_e$.
We want to use the lifts constructed in Definition \ref{def:dualnumberlifts} to define a map 
$f_{n,i}: Q_{e,0}\cup Q_{e,1}\to \mathrm{Mat}_{\ell_{n,i}}(k[[t_1,\ldots,t_n]])$
and an ideal $J_{n,i}$ of $k[[t_1,\ldots,t_n]]$ such that $J_{n,i}$ is the smallest ideal of $k[[t_1,\ldots,t_n]]$ with the property that $f_{n,i}$ defines a $k$-algebra 
homomorphism $\mathcal{N}\to  \mathrm{Mat}_{\ell_{n,i}}(k[[t_1,\ldots,t_n]]/J_{n,i})$. 
We first define certain matrices and determine their powers to set up the ideal $J_{n,i}$
(see also Definition \ref{def:needthis}(b)). 

\begin{dfn}
\label{def:ideal}
Fix a positive integer $n$, and let $N_n$ be the $n\times n$ matrix from Definition 
\ref{def:needthis}(b):
$$N_n=\left(\begin{array}{ccc|c}0&\cdots&0&t_n\\ \hline &&&t_{n-1}\\&\mathrm{I}_{n-1}&&\vdots\\&&&t_1\end{array}\right)$$
with entries in $k[t_1,\ldots,t_n]\subset k[[t_1,\ldots,t_n]]$. In particular, $N_1=(t_1)$.
Also define the following $(n+1)\times(n+1)$ matrix
$$\widetilde{N}_n=N_{n+1}\Big|_{t_{n+1}=0}=\left(\begin{array}{ccc|c}0&\cdots&0&0\\ \hline &&&t_n\\&\mathrm{I}_n&&\vdots\\&&&t_1\end{array}\right)\,.$$
Define inductively the following polynomials $h_{a,\nu}$ in $k[t_1,\ldots,t_n]\subset k[[t_1,\ldots,t_n]]$ for $1\le a\le n$, $\nu\ge 0$:
\begin{eqnarray}
\label{eq:polys0}
h_{a,0}&=& \left\{\begin{array}{ccl}
1&:&a=1\;,\\[2ex]
0&:&2\le a \le n\;;
\end{array}\right.\\[3ex]
h_{a,\nu} &=& \left\{\begin{array}{ccl}
t_n\, h_{n,\nu-1}&:&a=1\;,\\[2ex]
\label{eq:polys1}
h_{a-1,\nu-1} + t_{n-a+1}\,h_{n,\nu-1}&:&2\le a \le n\;.
\end{array}\right.
\end{eqnarray}
\end{dfn}

The following result is straightforward, so we omit its proof.
\begin{lemma}
\label{lem:needthiseasy}
Let $n$ be a positive integer, and let $N_n$ and $\widetilde{N}_n$ be as in
Definition $\ref{def:ideal}$. 
\begin{itemize}
\item[(i)] For $1\le a\le n$ and $\nu\ge 0$ , let $h_{a,\nu}\in k[t_1,\ldots,t_n]$ 
be as in Definition $\ref{def:ideal}$.
For all $\nu\ge 1$,
\begin{eqnarray*}
(N_n)^\nu &=& \left( \begin{array}{cccc} h_{1,\nu}&h_{1,\nu+1}&\cdots&h_{1,n+\nu-1}\\
h_{2,\nu}&h_{2,\nu+1}&\cdots&h_{2,n+\nu-1}\\ \vdots&\vdots&\ddots&\vdots\\
h_{n,\nu}&h_{n,\nu+1}&\cdots&h_{n,n+\nu-1}\\
\end{array}\right)\;;\\[3ex]
(\widetilde{N}_n)^\nu &=& 
 \left( \begin{array}{ccc|c} 0&\cdots&0&0\\ \hline &&&h_{1,n+\nu-1}\\
&(N_n)^{\nu-1}&&\vdots\\&&&h_{n,n+\nu-1}\\
\end{array}\right)\;.
\end{eqnarray*}
\vspace{1ex}
\item[(ii)] We have
\begin{eqnarray*}
(N_n)^n&=& t_n\,(\mathrm{I}_n) + t_{n-1}\,(N_n) + \cdots + t_2\, (N_n)^{n-2} + 
t_1\, (N_n)^{n-1} \;;\\[3ex]
(\widetilde{N}_n)^{n+1}&=& t_n\,(\widetilde{N}_n) + 
t_{n-1}\,(\widetilde{N}_n)^2 + \cdots+ t_2\, (\widetilde{N}_n)^{n-1} 
+ t_1\, (\widetilde{N}_n)^{n} \;.
\end{eqnarray*}
\end{itemize}
\end{lemma}

\begin{dfn}
\label{def:universallift}
Fix integers $1\le n\le \mu$ and $0\le i\le e-1$ such that $\ell_{n,i}=n\,e+i$ satisfies $\ell_{n,i}\le \lfloor \ell/2 \rfloor$.
Define $J_{n,i}$ to be the ideal of $k[[t_1,\ldots,t_n]]$ generated by the entries of the matrix $(N_n)^{m_i}$, where
$$m_i =\left\{\begin{array}{ccl}
\mu&:&0\le i \le \ell'\;,\\[2ex]
\mu-1&:&\ell'+1\le i \le e-1\;.
\end{array}\right.$$
By Definition \ref{def:ideal} and Lemma \ref{lem:needthiseasy}(i), we have
\begin{equation}
\label{eq:ideal!}
J_{n,i} = \left(h_{1,m_i},h_{2,m_i},\ldots , h_{n,m_i}\right)\,.
\end{equation}
Define a map
$$f_{n,i}:\qquad Q_{e,0}\cup Q_{e,1}\to \mathrm{Mat}_{\ell_{n,i}}(k[[t_1,\ldots,t_n]])$$
as follows. Viewing $k$ as a $k$-subalgebra of $k[[t_1,\ldots,t_n]]$ and using the notation introduced
in Definition \ref{def:notation}, define for all $1\le v\le e$:
\begin{eqnarray*}
f_{n,i}(v) &=& \rho_{n,i}(v)\quad\mbox{as in (\ref{eq:rho11vertices}), and}\\[2ex]
f_{n,i}(\alpha_v) &=& \left(\begin{array}{ccccc}0&0&\quad\cdots\quad&0&\delta_{v,e}B_{e,n,i}\\[1ex]
\delta_{v,1}B_{1,n,i}&0&&&0\\[1ex] 0&\delta_{v,2}B_{2,n,i}&\;\ddots&&\vdots\\[2ex]
\vdots&\ddots&\;\ddots&0&0\\[1ex] 0&\cdots&0&\delta_{v,e-1}B_{e-1,n,i}&0\end{array}\right)
\end{eqnarray*}
such that
\begin{equation}
\label{eq:Bvni}
B_{v,n,i}=\left\{\begin{array}{ccl}
A_{v,n,i}&:&v\not\equiv i\mod e\;,\\[2ex]
A_{v,n,i}+\left(\begin{array}{ccc|c}0&\cdots&0&t_n\\ \vdots&&\vdots&\vdots\\ 0&\cdots&0&t_1\end{array}\right)&:&v\equiv  i\mod e\;,
\end{array}\right.
\end{equation}
where $A_{v,n,i}$ is as in (\ref{eq:Avni}).
\end{dfn}

\begin{lemma}
\label{lem:ideal!!!}
Let $1\le n\le \mu$ and $0\le i\le e-1$ be integers such that $\ell_{n,i}=n\,e+i$ satisfies $\ell_{n,i}\le \lfloor \ell/2 \rfloor$.
Let $f_{n,i}$ and $J_{n,i}$ be as in Definition $\ref{def:universallift}$. Then $f_{n,i}$ defines a $k$-algebra
homomorphism
$$\rho_{U,n,i}:\qquad \mathcal{N}\to \mathrm{Mat}_{\ell_{n,i}}(k[[t_1,\ldots,t_n]]/J_{n,i})$$
and $J_{n,i}$ is the smallest ideal of $k[[t_1,\ldots,t_n]]$ with this property.
\end{lemma}

\begin{proof}
For each $v\in\{1,\ldots,e\}$, define $E_{v,n,i}$ to be the following product of $\ell=\mu\,e+\ell'$
matrices:
\begin{eqnarray*}
E_{v,n,i}&=&\left(f_{n,i}(\alpha_v)\cdot f_{n,i}(\alpha_{v-1})\cdots f_{n,i}(\alpha_1) \cdot f_{n,i}(\alpha_e)\cdots f_{n,i}(\alpha_{v+1})\right)^\mu\cdot\\
&&\,f_{n,i}(\alpha_v)\cdot f_{n,i}(\alpha_{v-1})\cdots f_{n,i}(\alpha_{v-\ell'+1})
\end{eqnarray*}
where  we take the indices of the $\alpha$'s occurring in the last $\ell'$ matrices modulo $e$ if necessary.
To prove that $f_{n,i}$ defines a $k$-algebra homomorphism modulo $J_{n,i}$ and that
$J_{n,i}$ is the smallest ideal of $k[[t_1,\ldots,t_n]]$ with this property, it suffices to show that $E_{v,n,i}$
has entries in $J_{n,i}$ for all $1\le v\le e$ and that there exists an element $v_0\in\{1,\ldots,e\}$ such that the entries of $E_{v_0,n,i}$
generate $J_{n,i}$. 

Fix $v\in\{1,\ldots,e\}$. Then $E_{v,n,i}$ is an $e\times e$ block matrix whose blocks are of the same size as
in $f_{n,i}(\alpha_v)$. Moreover, the only block that is not a zero matrix is the $(a_v,b_v)$ block where
$1\le a_v,b_v\le e$ and $a_v\equiv v+1 \mod e$ and $b_v\equiv v-\ell'+1 \mod e$. Letting
$C_{v,n,i}$ be the $(a_v,b_v)$ block of $E_{v,n,i}$, we obtain
$$C_{v,n,i}=\left(B_{v,n,i} B_{v-1,n,i}\cdots B_{1,n,i}B_{e,n,i}\cdots B_{v+1,n,i}\right)^\mu
B_{v,n,i}B_{v-1,n,i}\cdots B_{v-\ell'+1,n,i}$$
where we use the matrices defined in (\ref{eq:Bvni}) and we take the first indices of the last $\ell'$ matrices modulo 
$e$ if necessary.
Define $B_{v,n,i,0}=\mathrm{I}_{\theta(a_v,n,i)}$, and, for $1\le w\le e$, define $B_{v,n,i,w}=B_{v,n,i}B_{v-1,n,i}\cdots 
B_{v-w+1,n,i}$
to be a product of $w$ matrices, where we take again the first indices of these matrices modulo 
$e$ if necessary. Then we can write
$$C_{v,n,i} = \left(B_{v,n,i,e}\right)^\mu\, B_{v,n,i,\ell'}\;.$$
Note that
\begin{equation}
\label{eq:Bvnie}
B_{v,n,i,e}=\left\{\begin{array}{ccl} N_n&:&i=0 \mbox{ or } 1\le i \le v\le e-1\;;
\\[2ex]
\widetilde{N}_n&:&1\le v\le i-1  \mbox{ or } v=e \mbox{ and } i\ge 1 \;.
\end{array}\right.
\end{equation}

Suppose first that $i=0$. Then $C_{v,n,0}$ is equal to either $(N_n)^\mu$ or $(N_n)^{\mu+1}$. In particular,
if $v_0\in\{1,\ldots,e\}$ such that $v_0\equiv \ell'\mod e$ then $C_{v_0,n,0}=(N_n)^\mu$. It follows
from Definition \ref{def:universallift} and Lemma \ref{lem:needthiseasy}(i) 
that the entries of $E_{v_0,n,0}$ generate the same ideal as the entries of 
$(N_n)^{\mu}$ and that for all $v\in\{1,\ldots,e\}$, the entries of $E_{v,n,0}$ lie in this ideal. 
This proves Lemma \ref{lem:ideal!!!} for $i=0$.

Now suppose $i\ge 1$. If $i \le v\le e-1$, we have the following three possibilities for $C_{v,n,i}$:
\begin{equation}
\label{eq:poss1}
\left(N_n\right)^{\mu+1},\qquad 
\left(\begin{array}{ccc|c}&&&h_{1,\mu+n}\\&\left(N_n\right)^\mu&&\vdots\\&&&h_{n,\mu+n}\end{array}\right), \qquad
\left(N_n\right)^{\mu},
\end{equation}
where $\left(N_n\right)^{\mu+1}$ occurs precisely when $i\le v\le \ell'-1$, and
$\left(N_n\right)^{\mu}$ occurs precisely when $v\ge \max(i,\ell')$.

If $1\le v\le i-1$ or $v=e$, we have the following three possibilities for $C_{v,n,i}$:
\begin{equation}
\label{eq:poss2}
(\widetilde{N}_n)^{\mu+1},\qquad
\left(\begin{array}{ccc}0&\cdots&0\\\hline &&\\&\left(N_n\right)^\mu\\ &&\end{array}\right),\qquad
(\widetilde{N}_n)^\mu,
\end{equation}
where $(\widetilde{N}_n)^\mu$ occurs precisely when $\ell'\le v\le i-1$ or when $v=e$ and $\ell'=0$.

Let $v_0\in\{1,\ldots,e\}$ be such that $v_0\equiv \ell'\mod e$.
If $1\le i\le \ell'$, then $(\widetilde{N}_n)^\mu$ in (\ref{eq:poss2}) cannot occur. In this case,
it follows from (\ref{eq:poss1}), (\ref{eq:poss2}) and Lemma \ref{lem:needthiseasy}(i)
that the entries of $C_{v_0,n,i}$ generate the same ideal as the entries of 
$(N_n)^\mu$ and that for all $v\in\{1,\ldots,e\}$, the entries of $E_{v,n,i}$ lie in this ideal. 
On the other hand, if $\ell'+1\le i\le e-1$ then $(\widetilde{N}_n)^\mu$ in (\ref{eq:poss2}) can occur.
In this case, it follows from (\ref{eq:poss1}), (\ref{eq:poss2}) and Lemma \ref{lem:needthiseasy}(i)
that the entries of $C_{v_0,n,i}$ generate the same ideal as the entries of 
$(N_n)^{\mu-1}$ and that for all $v\in\{1,\ldots,e\}$, the entries of $E_{v,n,i}$ lie in this ideal. 
This proves Lemma \ref{lem:ideal!!!} in the case when $i\ge 1$.
\end{proof}

\begin{thm}
\label{thm:versallift}
Let $1\le n\le \mu$ and $0\le i\le e-1$ be integers such that $\ell_{n,i}=n\,e+i$ satisfies $\ell_{n,i}\le \lfloor \ell/2 \rfloor$.
Let $J_{n,i}$ be the ideal in $k[[t_1,\ldots,t_n]]$ from Definition $\ref{def:universallift}$ and let
$$\rho_{U,n,i}:\qquad \mathcal{N}\to \mathrm{Mat}_{\ell_{n,i}}(k[[t_1,\ldots,t_n]]/J_{n,i})$$
be the $k$-algebra homomorphism from Lemma $\ref{lem:ideal!!!}$.
The versal deformation ring $R(\mathcal{N},V_{n,i})$ of $V_{n,i}$ is isomorphic to 
$$R_{n,i}=k[[t_1,\ldots,t_n]]/J_{n,i}$$ 
with the reduction map $\pi_{R_{n,i}}:R_{n,i}\to k$ given by the morphism in $\hat{\mathcal{C}}$ sending $t_j$ to $0$ for $1\le j\le n$.
Moreover, the versal deformation of $V_{n,i}$ over $R_{n,i}$ is given by the strict 
equivalence class $[\rho_{U,n,i}]$.
\end{thm}

\begin{proof}
Let $S=R(\mathcal{N},V_{n,i})$ be the versal deformation ring of $V_{n,i}$, with reduction map 
$\pi_S:S\to k$. Since $\mathrm{dim}_k\,\mathrm{Ext}^1_\mathcal{N}(V_{n,i},V_{n,i})=n$
by Lemma \ref{lem:ext1}, it follows that $S$ is isomorphic to a quotient algebra of 
$k[[t_1,\ldots,t_n]]$ and that $n$ is minimal with this property. 
Let $\tau:\mathcal{N} \to \mathrm{Mat}_{\ell_{n.i}}(S)$ be 
a versal lift of $V_{n,i}$ over $S$ such that
$$\pi_S\circ \tau = \rho_{n,i}$$
where $\rho_{n,i}$ is the representation of $V_{n,i}$ from Definition \ref{def:notation}. 
Since by Lemma \ref{lem:ideal!!!}, $\rho_{U,n,i}$ is a lift of
$V_{n,i}$ over $R_{n,i}$, there exists a (not necessarily unique) morphism 
$$\gamma:S\to R_{n,i}$$
in $\hat{\mathcal{C}}$ such that we have an equality 
$$[\gamma \circ \tau] = [\rho_{U,n,i}]$$ of strict equivalence classes.
For $s\in\{1,\ldots,n\}$, let $\gamma_s:R_{n,i}\to R_s= k[[t_s]]/(t_s^2)$ be the morphism 
in $\hat{\mathcal{C}}$ sending $t_s$ to $t_s$ and $t_j$ to 0 for $1\le j \le n$, 
$j\neq s$. Then, for $1\le s\le n$, we have
$$[\gamma_s\circ\gamma\circ \tau] = [\gamma_s\circ \rho_{U,n,i}] = [\rho_{n,i,s}]$$ 
where $\rho_{n,i,s}$ is as in  Definition \ref{def:dualnumberlifts}. In other words,
using Lemma \ref{lem:ext2} together with Definition 
\ref{def:dualnumberlifts}, we obtain that $\lambda$
ranges over all morphisms $R_{n,i}\to k[\epsilon]$ in $\hat{\mathcal{C}}$
if and only if $\lambda\circ\gamma$
ranges over all morphisms $S\to k[\epsilon]$ in $\hat{\mathcal{C}}$. This implies that 
$\gamma:S \to R_{n,i}$ is surjective.

Suppose now that $\gamma$ is not injective. Then there must exist a non-trivial lift 
$\tau':\mathcal{N}\to \mathrm{Mat}_{\ell_{n,i}}(S')$ of $\rho_{U,n,i}$ over a ring
of the form $S'=k[[t_1,\ldots,t_n]]/J'$, corresponding to a surjective morphism 
$\gamma':S'\to R_{n,i}$ in $\hat{\mathcal{C}}$ with a non-trivial kernel, such that
$$[\gamma'\circ \tau'] = [\rho_{U,n,i}]\;.$$
Let $\gamma'': S \to S'$ be a morphism in $\hat{\mathcal{C}}$ such that
$$[\gamma'' \circ \tau] = [\tau']\;.$$
In particular, we have
$$[(\gamma'\circ\gamma'')\circ \tau] = [\rho_{U,n,i}]\;.$$
Using the same argument as above, we obtain that $\gamma'\circ\gamma''$ is surjective. Note that $\gamma'\circ\gamma''$ may in principle be different from $\gamma$, 
since we have not proved yet that $V_{n,i}$ has a universal deformation ring, but we only know that it has a versal deformation ring. Since $R_{n,i}$ has finite $k$-dimension, 
we know, however, that $\gamma'\circ\gamma''$ is injective if and only if 
$\mathrm{dim}_k\, S = \mathrm{dim}_k\, R_{n,i}$ if and only if
$\gamma$ is injective. Hence we can (and will) assume in what follows that 
$\gamma=\gamma'\circ\gamma''$.

For $1\le j \le n$, choose an element $u_j\in\gamma'^{-1}(t_j)$. Since $\gamma'$ is a morphism in $\hat{\mathcal{C}}$, it satisfies $\gamma'^{-1}(\mathfrak{m}_{R_{n,i}}) = \mathfrak{m}_{S'}$.
This means that $u_1,\ldots,u_n$ generate the maximal ideal $\mathfrak{m}_{S'}$, which implies that the morphism $\sigma:S'\to S'$ in $\hat{\mathcal{C}}$, defined by
$\sigma(t_j)=u_j$ for $1\le j\le n$, is an isomorphism. Define 
$\widetilde{\gamma}'=\gamma'\circ\sigma: S'\to R_{n,i}$, define 
$\widetilde{\gamma}'' = \sigma^{-1}\circ\gamma'':S \to S'$,
and define
$$\widetilde{\tau}' = \widetilde{\gamma}''\circ \tau:
\quad \mathcal{N}\to \mathrm{Mat}_{\ell_{n,i}}(S')\;.$$
Note that $\widetilde{\gamma}'(t_j)=t_j$ for $1\le j\le n$, meaning that 
$\widetilde{\gamma}':S'\to R_{n,i}$ is the natural projection. In particular, $J'$ is properly contained in $J_{n,i}$, and 
$\mathrm{Ker}(\widetilde{\gamma}' ) = J_{n,i}/J'$. Also, note that 
$\gamma =\widetilde{\gamma}'\circ \widetilde{\gamma}''$ and that $\widetilde{\tau}'$ 
is a non-trivial lift of $\rho_{U,n,i}$ over $S'$. 
We now show that $\widetilde{\tau}'$ does not exist, which implies that $\gamma$ is injective. To prove this, we can restrict to the case when 
\begin{equation}
\label{eq:assume1}
J'\supseteq (t_1,\ldots,t_n)\, J_{n,i}\;.
\end{equation}
Hence, we assume this from now on.

Since $[\widetilde{\gamma}'\circ \widetilde{\tau}'] = [\rho_{U,n,i}]$ and since 
$\widetilde{\gamma}':S'\to R_{n,i}$ is the natural projection,
there exists a matrix $\Sigma'$ in $\mathrm{Mat}_{\ell_{n,i}}(S')$ which is congruent to
the identity matrix modulo $\mathfrak{m}_{S'}$ such that 
$$\rho_{U,n,i} =
\widetilde{\gamma}'(\Sigma')\, \left(\widetilde{\gamma}'\circ \widetilde{\tau}'\right)\,
\widetilde{\gamma}'(\Sigma')^{-1} 
=\widetilde{\gamma}'\circ \left(\Sigma'\,\widetilde{\tau}'\,{\Sigma'}^{-1}\right)\,.$$
Replacing $\widetilde{\tau}'$ by $\Sigma'\,\widetilde{\tau}'\,{\Sigma'}^{-1}$, we can
(and will) assume from now on that
$$\widetilde{\gamma}'\circ \widetilde{\tau}'=\rho_{U,n,i}\;.$$
This means that for $1\le v\le e$, we can write
\begin{equation}
\label{eq:assume2}
\widetilde{\tau}'(\alpha_v)= f_{n,i}(\alpha_v) + \widetilde{D}_v\quad \mod J'
\end{equation}
for an $e\times e$ block matrix 
$$\widetilde{D}_v=\left(\widetilde{D}_{v,a,b}\right)_{1\le a,b\le e}$$
with entries in $k[[t_1,\ldots,t_n]]$,
where $\widetilde{D}_{v,a,b}$ is a $\theta(a,n,i)\times \theta(b,n,i)$ matrix and
$\theta(a,n,i)$ and $\theta(b,n,i)$ are as in (\ref{eq:delta}).
Moreover, $\widetilde{D}_{v,a,b}$ has entries in $J_{n,i}$. Since $\widetilde{\tau}'$
is a $k$-algebra homomorphism, we must have that for all $v\in\{1,\ldots,e\}$
the product of $\ell=\mu\,e+\ell'$ matrices
\begin{equation}
\label{eq:matrixproduct}
\left(\widetilde{\tau}'(\alpha_v)\cdot \widetilde{\tau}'(\alpha_{v-1})\cdots 
\widetilde{\tau}'(\alpha_1)\cdot \widetilde{\tau}'(\alpha_e)\cdots
\widetilde{\tau}'(\alpha_{v+1})\right)^\mu\cdot
\widetilde{\tau}'(\alpha_v)\cdot \widetilde{\tau}'(\alpha_{v-1})\cdots 
\widetilde{\tau}'(\alpha_{v-\ell'+1})
\end{equation}
lies in $\mathrm{Mat}_{\ell_{n,i}}(J')$ for all $v\in\{1,\ldots,e\}$, 
where  we take the indices of the $\alpha$'s occurring in the last $\ell'$ matrices modulo $e$ if necessary.
Using (\ref{eq:assume2}), we can expand this matrix product and write it as a sum of
monomials in $f_{n,i}(\alpha_w)$ and $\widetilde{D}_{w'}$, for $1\le w,w'\le e$.
By (\ref{eq:assume1}), since $J_{n,i}\subset (t_1,\ldots, t_n)$, it follows that any 
such monomial involving at least two matrices $\widetilde{D}_w$ and $\widetilde{D}_{w'}$ is a matrix with
entries in $J'$. 
To consider monomials involving precisely one matrix $\widetilde{D}_w$, we write
$$F_{v,n,i}=f_{n,i}(\alpha_v)\cdot f_{n,i}(\alpha_{v-1})\cdots f_{n,i}(\alpha_1) 
\cdot f_{n,i}(\alpha_e)\cdots f_{n,i}(\alpha_{v+1})\;.$$
Since $\ell_{n,i}\le \lfloor \ell/2 \rfloor$, it follows that 
$$2n\,e + 2i\le \mu\, e + \ell'\,,$$
and hence either $\mu > 2n+1$, or $\mu=2n+1$ and $e+\ell'\ge 2i$, or
$\mu=2n$ and $\ell'\ge 2i$. We have the following possibilities to consider for
the monomials in (\ref{eq:matrixproduct}) involving precisely one matrix $\widetilde{D}_w$:
\vspace{1ex}
\begin{enumerate}
\item[(A)] $\cdots \widetilde{D}_w\cdots \left(F_{v,n,i}\right)^{n+1}\cdots$ , or
\vspace{1ex}
\item[(B)] $\cdots \left(F_{v,n,i}\right)^{n+1}\cdots \widetilde{D}_w\cdots$ , or
\vspace{1ex}
\item[(C)]$\left(F_{v,n,i}\right)^{n}\cdots \widetilde{D}_w\cdots\left(F_{v,n,i}\right)^{n}
\cdot f_{n,i}(\alpha_v)\cdot  f_{n,i}(\alpha_{v-1})\cdots  f_{n,i}(\alpha_{v-\ell'+1})$  and $\mu=2n+1$, or
\vspace{1ex}
\item[(D)]$\left(F_{v,n,i}\right)^{n}\cdots \widetilde{D}_w\cdots\left(F_{v,n,i}\right)^{n-1}
\cdot f_{n,i}(\alpha_v)\cdot  f_{n,i}(\alpha_{v-1})\cdots  f_{n,i}(\alpha_{v-\ell'+1})$  and $\mu=2n$, or
\vspace{1ex}
\item[(E)] $\left(F_{v,n,i}\right)^{n-1} \cdots \widetilde{D}_w\cdots \left(F_{v,n,i}\right)^{n}\cdot f_{n,i}(\alpha_v)\cdot  f_{n,i}(\alpha_{v-1})\cdots  f_{n,i}(\alpha_{v-\ell'+1})$ and $\mu = 2n$.
\end{enumerate}
\vspace{1ex}
As in the proof of Lemma \ref{lem:ideal!!!}, we see that $F_{v,n,i}$ is an $e\times e$ block matrix whose
blocks are of the same size as in $f_{n,i}(\alpha_v)$. Moreover, the only block that is not a zero matrix is the
$(a,a)$ block for $a \equiv (v+1)\mod e$ and this block is equal to the matrix $B_{v,n,i,e}$ from (\ref{eq:Bvnie}).

By Lemma \ref{lem:needthiseasy}(i), it follows that $(F_{v,n,i})^{n+1}$ always has entries in $(t_1,\ldots,t_n)$.
By (\ref{eq:assume1}), this means that the matrix products in the cases (A) and (B)
always have entries in $J'$. If $i=0$ or $1\le i \le v\le e-1$ then $B_{v,n,i,e}=N_n$. Hence it also follows that the 
matrix products in the cases (C), (D) and (E) have entries in $J'$. Thus we need to discuss
the cases (C), (D) and (E) when $1\le v\le i-1$ or $v=e$ and $i\ge 1$.

Suppose $1\le v\le i-1$ or $v=e$ and $i\ge 1$. 
If $w\le e-1$ then the matrix products in the cases (C) and (D) have the form
\begin{equation}
\label{eq:need11}
\left(F_{v,n,i}\right)^{n}\cdots f_{n,i}(\alpha_e)\cdots  \widetilde{D}_w\cdots .
\end{equation}
If $w=e$ then the matrix product in  (C) has the form
\begin{equation}
\label{eq:need12}
\left(F_{v,n,i}\right)^{n}\cdots \widetilde{D}_e\cdots  f_{n,i}(\alpha_i)\cdots \left(F_{v,n,i}\right)^{n}\cdots
\end{equation}
and the matrix product in  (D) has the form
\begin{equation}
\label{eq:need12A}
\left(F_{v,n,i}\right)^{n}\cdots \widetilde{D}_e \cdots  f_{n,i}(\alpha_i)\cdots 
\left(F_{v,n,i}\right)^{n-1}\cdots f_{n,i}(\alpha_e) \cdots 
\end{equation}
On the other hand, the matrix product in (E) has the form
\begin{equation}
\label{eq:need13}
\left(F_{v,n,i}\right)^{n-1}\cdots \widetilde{D}_w\cdots \left(F_{v,n,i}\right)^{n}
\cdots f_{n,i}(\alpha_e)\cdots  
\end{equation}
Using the matrices defined in (\ref{eq:Bvni}) and (\ref{eq:Bvnie}), we see that the matrix products
$$(B_{v,n,i,e})^{n}B_{e,n,i}\;,\quad B_{i,n,i}(B_{v,n,i,e})^{n}\;,\quad 
B_{i,n,i}(B_{v,n,i,e})^{n-1}B_{e,n,i}\;,\quad\mbox{and}\quad (B_{v,n,i,e})^{n}B_{e,n,i}$$
all have entries in $(t_1,\ldots,t_n)$. Hence the product of matrices to the left of $\widetilde{D}_w$ in (\ref{eq:need11})
and to the right of $\widetilde{D}_e$ (resp. $\widetilde{D}_w$) in (\ref{eq:need12}) and (\ref{eq:need12A})
(resp. in (\ref{eq:need13}))
has entries in $(t_1,\ldots,t_n)$. By (\ref{eq:assume1}), it follows that the matrix products in (\ref{eq:need11}),
(\ref{eq:need12}), (\ref{eq:need12A}) and (\ref{eq:need13}) all have entries in $J'$.

We conclude that for all $v\in\{1,\ldots,e\}$, each monomial in the matrix product
(\ref{eq:matrixproduct}) involving precisely one matrix $\widetilde{D}_w$ is a matrix with
entries in $J'$. This means that for all $v\in\{1,\ldots,e\}$, the
matrix product (\ref{eq:matrixproduct}) and the matrix product of $\ell=\mu\, e +\ell'$
matrices
\begin{equation}
\label{eq:another1}
\left(F_{v,n,i}\right)^\mu\cdot
f_{n,i}(\alpha_v)\cdot f_{n,i}(\alpha_{v-1})\cdots 
f_{n,i}(\alpha_{v-\ell'+1})
\end{equation}
are congruent modulo $\mathrm{Mat}_{\ell_{n,i}}(J')$. Since the matrix product 
(\ref{eq:matrixproduct}) lies in $\mathrm{Mat}_{\ell_{n,i}}(J')$ for all $v\in\{1,\ldots,e\}$, 
it follows that the matrix product (\ref{eq:another1}) also lies in 
$\mathrm{Mat}_{\ell_{n,i}}(J')$ for all $v\in\{1,\ldots,e\}$.
Let $v_0\in\{1,\ldots,e\}$ be such that $v_0\equiv \ell'\mod e$.
Arguing the same way as in the proof of Lemma
\ref{lem:ideal!!!}, we see that for $v=v_0$, the entries of the
matrix product in (\ref{eq:another1}) generate $J_{n,i}$. But this means that
$J_{n,i}=J'$, which is a contradiction to our assumption that $J'$ is properly contained
in $J_{n,i}$. Therefore, the lift $\widetilde{\tau}'$ does not exist, which implies that
$\gamma$ is an isomorphism in $\hat{\mathcal{C}}$. This
implies that $S=R(\mathcal{N}, V_{n,i})\cong R_{n,i}$ and that
the versal deformation of $V_{n,i}$ over $R_{n,i}$ is given by the strict equivalence class 
$[\rho_{U,n,i}]$.
\end{proof}

We next prove that for $n$ and $i$ as in Theorem \ref{thm:versallift},
the ring $R_{n,i}$ is a universal deformation ring of $V_{n,i}$ by proving that
the deformation functor $\mathrm{Def}_{\mathcal{N}}(V_{n,i},-)$ has the centralizer 
lifting property (see Definition \ref{def:centralizerlifting}).

\begin{thm}
\label{thm:universal}
Let $1\le n\le \mu$ and $0\le i\le e-1$ be integers such that $\ell_{n,i}=n\,e+i$ satisfies 
$\ell_{n,i}\le \lfloor \ell/2 \rfloor$. Let $J_{n,i}$, $R_{n,i}$ and $\rho_{U,n,i}$ be as in 
Theorem $\ref{thm:versallift}$. Then $R_{n,i}$ is a universal deformation ring
of $V_{n,i}$ and the universal deformation of $V_{n,i}$ over $R_{n,i}$ is given by
the strict equivalence class $[\rho_{U,n,i}]$.
\end{thm}

\begin{proof}
Let $\rho_{n,i}:\mathcal{N}\to\mathrm{Mat}_{\ell_{n,i}}(k)$ be the representation of 
$V_{n,i}$ from Definition \ref{def:notation}, let $R_{n,i}=k[[t_1,\ldots,t_n]]/J_{n,i}$, and let
$$\rho_{U,n,i}:\qquad \mathcal{N}\to \mathrm{Mat}_{\ell_{n,i}}(R_{n,i})$$
be the $k$-algebra homomorphism from Lemma $\ref{lem:ideal!!!}$.
Let $R$ be an arbitrary ring in $\hat{\mathcal{C}}$, let $\gamma:R_{n,i}\to R$
be a morphism in $\hat{\mathcal{C}}$, and define 
$\tau_\gamma:\mathcal{N}\to \mathrm{Mat}_{\ell_{n,i}}(R)$ by 
$$\tau_\gamma = \gamma\circ \rho_{U,n,i}\;.$$
More precisely, for $1\le j\le n$, define $r_j=\gamma(t_j)\in R$.
Let $\hat{\gamma}:k[[t_1,\ldots,t_n]]\to R$ be the morphism in $\hat{\mathcal{C}}$
defined by $\hat{\gamma}(t_j)=r_j$ for $1\le j\le n$. Then, for all $v\in\{1,\ldots,e\}$, we have
$$\tau_\gamma(v)= \hat{\gamma}(f_{n,i}(v))\qquad\mbox{and}\qquad
\tau_\gamma(\alpha_v)= \hat{\gamma}(f_{n,i}(\alpha_v))$$
where $f_{n,i}:Q_{e,0}\cup Q_{e,1}\to \mathrm{Mat}_{\ell_{n,i}}(k[[t_1,\ldots,t_n]])$ is as in Definition \ref{def:universallift}. 

We first determine the set
$$Z_{\gamma}=\{\Sigma\in \mathrm{Mat}_{\ell_{n,i}}(R)\;|\;\Sigma\,\tau_\gamma=\tau_\gamma\,\Sigma\}\;.$$
Each matrix $\Sigma$ in $Z_{\gamma}$ is an $e\times e$ block matrix
$\Sigma=\left(\Sigma_{a,b}\right)_{1\le a,b\le e}$ such that 
$\Sigma_{a,b}$ is a $\theta(a,n,i)\times \theta(b,n,i)$ matrix and 
$\theta(a,n,i)$ and $\theta(b,n,i)$ are as in (\ref{eq:delta}).

For $1\le v\le e$, the condition $\Sigma\;\tau_\gamma(v)=\tau_\gamma(v)\,\Sigma$
means that $\Sigma_{a,b}$ is the zero matrix for $a\neq b$.
For $1\le v\le i-1$ or $i+1\le v\le e-1$, the condition
$\Sigma\;\tau_\gamma(\alpha_v)=\tau_\gamma(\alpha_v)\,\Sigma$
additionally means that 
$\Sigma_{v,v}=\Sigma_{v+1,v+1}$. In other words,
$$\Sigma_{1,1}=\cdots=\Sigma_{i,i}\quad\mbox{and}\quad  \Sigma_{i+1,i+1}=\cdots=\Sigma_{e,e}\;.$$
If $i=0$ then the condition $\Sigma\,\tau_\gamma(\alpha_e)=\tau_\gamma(\alpha_e)\,\Sigma$ additionally means
that 
\begin{equation}
\label{eq:i0condition2}
\Sigma_{1,1}\,\hat{\gamma}(N_n) = \hat{\gamma}(N_n)\,\Sigma_{e,e} \;.
\end{equation}
If $i\ge 1$ then the conditions $\Sigma\,\tau_\gamma(\alpha_e)=\tau_\gamma(\alpha_e)\,\Sigma$ 
and $\Sigma\,\tau_\gamma(\alpha_i)=\tau_\gamma(\alpha_i)\,\Sigma$ additionally mean
that 
\begin{equation}
\label{eq:11entry}
\Sigma_{1,1} = \left(\begin{array}{c|ccc}c_{1,1}&0&\cdots&0\\ \hline c_{2,1}&&&\\ \vdots&&\Sigma_{e,e}&\\
c_{n+1,1}&&&\end{array}\right)
\end{equation}
for appropriate elements $c_{1,1},\ldots,c_{n+1,1}$ 
in $R$, and that
\begin{equation}
\label{eq:i1condition2A}
\Sigma_{1,1}\,\hat{\gamma}(\widetilde{N}_n) = \hat{\gamma}(\widetilde{N}_n)\,\Sigma_{1,1} \;.
\end{equation}
Recall that $\theta(1,n,i)=n$ if $i=0$ and
that $\theta(1,n,i)=n+1$ if $i\ge 1$. Define $M_n=\hat{\gamma}(N_n)$ when $i=0$, and define 
$M_{n+1}=\hat{\gamma}(\widetilde{N}_n)$ when $i\ge 1$. Then (\ref{eq:i0condition2}) for $i=0$ (resp. 
(\ref{eq:i1condition2A}) for $i\ge 1$) is the same as
\begin{equation}
\label{eq:together}
\Sigma_{1,1}\,M_{\theta(1,n,i)} = M_{\theta(1,n,i)}\,\Sigma_{1,1} \;.
\end{equation}
Write the column vectors of $\Sigma_{1,1}$ as $\vec{c}_{1},\ldots,\vec{c}_{\theta(1,n,i)}$.
Comparing the left and right hand sides of (\ref{eq:together}) column by column, we obtain the following conditions:
\begin{eqnarray}
\label{eq:newcond1}
M_{\theta(1,n,i)}\,
\vec{c}_b&=& 
\vec{c}_{b+1}
\qquad \mbox{for }
1\le b \le \theta(1,n,i)-1,\\[2ex]
\label{eq:newcond2}
M_{\theta(1,n,i)}\,
\vec{c}_{\theta(1,n,i)}&=& 
r_{\theta(1,n,i)}\,\vec{c}_1 + r_{\theta(1,n,i)-1}\,\vec{c}_2+\cdots + r_1\,\vec{c}_1
\;,
\end{eqnarray}
where we define $r_{n+1}=0$.
Using induction, we see that (\ref{eq:newcond1}) is equivalent to the condition
\begin{equation}
\label{eq:condfinal1}
\vec{c}_{b}\;=\;
\left(M_{\theta(1,n,i)}\right)^{b-1}\, 
\vec{c}_1
\qquad\mbox{for } 1\le b\le \theta(1,n,i).
\end{equation}
In other words, the second column through the last column of $\Sigma_{1,1}$
can be obtained from its first column by multiplying by an appropriate power of $M_{\theta(1,n,i)}$. Substituting
(\ref{eq:condfinal1}) into (\ref{eq:newcond2}) and using that $M_{\theta(1,n,i)}=\hat{\gamma}(N_n)$ when
$i=0$ and that $r_{n+1}=0$ and $M_{\theta(1,n,i)}=\hat{\gamma}(\widetilde{N}_n)$ when 
$i\ge 1$, we obtain by Lemma \ref{lem:needthiseasy}(ii) that  (\ref{eq:newcond2}) follows from 
(\ref{eq:newcond1}). 

Given a column vector $\vec{c}$ in $R^{\theta(1,n,i)}$, we define the following two matrices:
\begin{itemize}
\item $M(\vec{c})$ is the $\theta(1,n,i)\times \theta(1,n,i)$ matrix whose $b$-th column vector is
equal to $\left(M_{\theta(1,n,i)}\right)^{b-1}\,\vec{c}\,$ for $1\le b\le \theta(1,n,i)$;
\item $M'(\vec{c})$ is the $(\theta(1,n,i)-1)\times (\theta(1,n,i)-1)$ matrix obtained from $M(\vec{c})$ by deleting its first row and first column.
\end{itemize}

Summarizing the above arguments, we obtain that 
$\Sigma=\left(\Sigma_{a,b}\right)_{1\le a,b\le e}$ lies in $Z_\gamma$ if and only if
there exists a column vector $\vec{c}\in R^{\theta(1,n,i)}$
such that 
\begin{equation}
\label{eq:final}
\left\{
\begin{array}{rcll}
\Sigma_{a,b}&=&\mbox{(zero matrix)} &\mbox{ if $a\ne b$,
}\\[2ex]
\Sigma_{a,a}&=&M(\vec{c})&\mbox{ if $i=0$ or $i\ge 1$ and $1\le a\le i$;}\\[2ex]
\Sigma_{a,a}&=&M'(\vec{c})&\mbox{ if $i\ge 1$ and $i+1\le a\le e$.}
\end{array}\right.
\end{equation}

As in Definition \ref{def:centralizerlifting}, define
$G_{R}=\mathrm{Ker}(\mathrm{GL}_{\ell_{n,i}}(R)\xrightarrow{\pi_{R}} \mathrm{GL}_{\ell_{n,i}}(k))$.
In other words, $G_R$ consists of all the matrices in $\mathrm{GL}_{\ell_{n,i}}(R)$ that are congruent to the identity 
matrix modulo $\mathfrak{m}_R$. Then
$\Sigma$ lies in $Z_\gamma\cap G_R$ if and only if $\Sigma$ satisfies the conditions in (\ref{eq:final})
and, additionally, the entries $c_1,c_2,\ldots,c_{\theta(1,n,i)}$ of $\vec{c}$
satisfy
\begin{equation}
\label{eq:cond_c_d}
c_1-1\in \mathfrak{m}_R,\qquad c_j\in \mathfrak{m}_R, \quad 2\le j\le \theta(1,n,i)\;.
\end{equation}

We next use the above analysis of $Z_\gamma$ and $Z_\gamma\cap G_R$ to prove that the deformation functor 
$\mathrm{Def}_{\mathcal{N}}(V_{n,i},-)$ has the centralizer lifting property. 
Let $A_1,A_0$ be Artinian rings in $\mathcal{C}$, and let $\alpha:A_1\to A_0$
be a morphism in $\mathcal{C}$ that is surjective. 
Note that $\alpha$ induces a surjective homomorphism $G_{A_1}\to G_{A_0}$.
Suppose $\tau_1:\mathcal{N}\to\mathrm{Mat}_{\ell_{n,i}}(A_1)$ is a lift of the representation $\rho_{n,i}$ of $V_{n,i}$
over $A_1$, and define $\tau_0=\alpha\circ\tau_1$. Since $[\rho_{U,n,i}]$ is a versal deformation of $\rho_{n,i}$ over
the versal deformation ring $R_{n,i}$, there exists a morphism
$\gamma_1: R_{n,i} \to A_1$ in $\hat{\mathcal{C}}$ such that
$$[\gamma_1\circ\rho_{U,n,i}] = [\tau_1]\;.$$
In other words, there exists a matrix $\Upsilon_1\in G_{A_1}$ such that
$$\gamma_1\circ\rho_{U,n,i} = \Upsilon_1 \, \tau_1 \,\Upsilon_1^{-1}\;.$$
Define $\Upsilon_0=\alpha(\Upsilon_1)$ and $\gamma_0=\alpha\circ\gamma_1$. 
Then $\Upsilon_0\in G_{A_0}$, and
$$\gamma_0\circ\rho_{U,n,i} = \Upsilon_0 \, \tau_0 \,\Upsilon_0^{-1}\;.$$
For $j\in\{0,1\}$ define $Z(\tau_j)=\{\Sigma_j\in G_{A_j}\;|\;\Sigma_j\,\tau_j=\tau_j\,\Sigma_j\}$.
Since $\tau_{\gamma_j}=\gamma_j\circ\rho_{U,n,i}$ in the definition of $Z_{\gamma_j}$, 
we have the equality
\begin{equation}
\label{eq:oy1}
Z(\tau_j)=\Upsilon_j^{-1}\,\left(Z_{\gamma_j}\cap G_{A_j}\right)\,\Upsilon_j\;.
\end{equation}
To prove that $\mathrm{Def}_{\mathcal{N}}(V_{n,i},-)$ has the centralizer lifting property, we need to show
that the natural homomorphism $Z(\tau_1)\to Z(\tau_0)$ induced by $\alpha$ is surjective. 
By our analysis of $Z_{\gamma_j}\cap G_{A_j}$ for $j\in\{0,1\}$ above, it follows that the natural homomorphism $Z_{\gamma_1}\cap G_{A_1}\to
Z_{\gamma_0}\cap G_{A_0}$ induced by $\alpha$ is surjective. Since $\Upsilon_0=\alpha(\Upsilon_1)$, the equality (\ref{eq:oy1}) 
then implies that the natural homomorphism $Z(\tau_1)\to Z(\tau_0)$ induced by $\alpha$ is surjective.
This completes the proof of Theorem \ref{thm:universal}.
\end{proof}

\medskip

\noindent
\textit{Proof of Theorem $\ref{thm:supermain}$.}
As in (\ref{eq:Nsimplify}), define $\mathcal{N}=\mathcal{N}(e,\ell)=k\,Q_e/J^\ell$. 
Write $\ell = \mu\, e + \ell'$ as in (\ref{eq:ell}).
Suppose $V$ is a finitely generated indecomposable non-projective $\mathcal{N}$-module, and
let $\ell_V = \mathrm{min}\{\mathrm{dim}_k\, V, \ell-\mathrm{dim}_k \,V \}$ be as in Theorem \ref{thm:supermain}.
Since $R(\mathcal{N},V)\cong R(\mathcal{N},\Omega(V))$ and since $R(\mathcal{N},V)$ is 
universal if and only if $R(\mathcal{N},\Omega(V))$ is universal, we can replace $V$ by 
$\Omega(V)$, if necessary, to be able to assume that $\ell_V=\mathrm{dim}_k\, V$. By taking a cyclic permutation of
the vertices $1,\ldots, e$ of the quiver $Q_e$ of $\mathcal{N}$, if necessary, we can also assume
that the radical quotient of $V$ is isomorphic to the simple $\mathcal{N}$-module corresponding to the vertex $1$.
Writing $\ell_V=n\, e + i$ as in (\ref{eq:ellV}), this means that, comparing the notation
in Theorem \ref{thm:supermain} with the notation in Definition \ref{def:notation}, we  have
$$V=V_{n,i}\quad \mbox{and} \quad \ell_V=\ell_{n,i}\le \lfloor \ell/2 \rfloor\;.$$

Suppose first that $n=0$. By Lemma \ref{lem:ext1}, $\mathrm{Ext}^1_{\mathcal{N}}(V,V)=0$. Thus, Remark 
\ref{rem:superlame} implies that the versal deformation ring of $V$ is universal and 
isomorphic to $k$.
Since $J_{0}(a)$ is the zero ideal of $k$ for all $a\ge 0$ by Definition \ref{def:needthis}(b),
Theorem \ref{thm:supermain} follows when $n=0$.

Suppose now that $1\le n\le \mu$. Defining $m_V$ as in (\ref{eq:mV}), this means that, comparing the notation
in Theorem \ref{thm:supermain} with the notation in Definition \ref{def:universallift}, we additionally have
$$m_V=m_{i}\quad \mbox{and} \quad J_{n,i}=J_{n}(m_V)\;.$$
Therefore, Theorem \ref{thm:supermain} follows from Theorem \ref{thm:universal} when $n\ge 1$. \hfill $\Box$

%%%%%%%%%%%%%%%%%%%%%%%%%%%%%%%%%%%%%%%%%%%%%%%%%%%%%%%%%
%% Stable equivalences of Morita type
%%%%%%%%%%%%%%%%%%%%%%%%%%%%%%%%%%%%%%%%%%%%%%%%%%%%%%%%%

\section{Stable equivalences of Morita type}
\label{s:stableMorita}
\setcounter{equation}{0}

The goal of this section is to prove Theorem \ref{thm:main2} and Corollary \ref{cor:supermain}.
As before, assume $k$ is an arbitrary field. Let $\Lambda$ and $\Gamma$ be two finite dimensional $k$-algebras.

Following Brou\'{e} \cite{broue1}, we say that there is a \emph{stable equivalence of Morita type} between 
$\Lambda$ and $\Gamma$ if there exist $X$ and $Y$ such that
$X$ is a $\Gamma$-$\Lambda$-bimodule and $Y$ is a $\Lambda$-$\Gamma$-bimodule, 
$X$ and $Y$ are projective both as left and as right modules, and 
we have the following isomorphisms
\begin{eqnarray}
\label{eq:stab}
Y\otimes_{\Gamma}X&\cong& \Lambda\oplus P \quad\mbox{ as $\Lambda$-$\Lambda$-bimodules, and} \\ 
\nonumber
X\otimes_{\Lambda}Y&\cong& \Gamma\oplus Q \quad\mbox{ as $\Gamma$-$\Gamma$-bimodules}, 
\end{eqnarray}
where $P$ is a projective $\Lambda$-$\Lambda$-bimodule, and $Q$ is a projective 
$\Gamma$-$\Gamma$-bimodule.

In particular, $X\otimes_{\Lambda}-$ and $Y\otimes_{\Gamma}-$ induce mutually inverse 
equivalences between the stable module categories $\Lambda\mbox{-\underline{mod}}$ and 
$\Gamma\mbox{-\underline{mod}}$.

\medskip

\textit{Proof of Theorem $\ref{thm:main2}$.}
Let $\Lambda$ be an indecomposable finite dimensional  $k$-algebra such that there exists
a stable equivalence of Morita type between $\Lambda$ and a  self-injective split basic Nakayama
algebra $\mathcal{N}$ over $k$. Suppose $V$ is a finitely generated indecomposable $\Lambda$-module.
If $V$ is projective, it follows from Remark \ref{rem:superlame} that $R(\Lambda,V)$ is universal and 
isomorphic to $k$, which proves part (i).

Suppose now that $V$ is not projective, and let $L$ be the Loewy length of $\Lambda$. If $L=1$ then all $\Lambda$-modules
are projective. Hence $L\ge 2$.

Suppose first that $L=2$. By \cite[Cor. 1.2 and Thm. 2.3]{Reiten}, it follows that $\Lambda$ is a 
Nakayama algebra.
Hences $V$ is a simple non-projective $\Lambda$-module. If 
$\mathrm{Ext}^1_\Lambda(V,V)=0$ then it follows from Remark \ref{rem:superlame} that $R(\Lambda,V)$ 
is universal and isomorphic to $k$. Suppose now that $\mathrm{Ext}^1_\Lambda(V,V)\neq 0$. Then the
projective cover $P_V$ of $V$ is a uniserial $\Lambda$-module of length 2, with composition factors
$V,V$. Since $\Lambda$ is indecomposable, this implies that, up to isomorphism,
there is a unique indecomposable projective $\Lambda$-module and a unique simple $\Lambda$-module
(see, for example, \cite[Prop. II.5.2]{ars}). But then $\Lambda$ is self-injective and the stable Auslander-Reiten
quiver $\Gamma_s(\Lambda)$ is a single vertex with no arrows. Since the stable Auslander-Reiten quivers
of $\Lambda$ and $\mathcal{N}$ are isomorphic as valued quivers, it follows that $\mathcal{N}\cong
\mathcal{N}(1,2)$ and that $V$ corresponds to a simple $\mathcal{N}$-module $S$, which is unique up to
isomorphism. Since $R(\mathcal{N},S)$ is universal and isomorphic to $k[[t]]/(t^2)$ by Theorem \ref{thm:supermain}, 
it follows by \cite[Prop. 3.2.6]{blehervelezderived} that $R(\Lambda,V)$ is also universal and isomorphic to 
$k[[t]]/(t^2)$, which proves part (ii).

Finally, suppose $L\ge 3$. By \cite[Thm. 2.4]{Reiten}, it follows that $\Lambda$ is self-injective. By 
\cite[Prop. X.1.8]{ars}, $\Lambda$ has no almost split sequences with projective middle terms.
Since $\Lambda$ has finite representation type, its Auslander-Reiten quiver is connected.
Therefore, it follows that the stable Auslander-Reiten quiver $\Gamma_s(\Lambda)$ is also connected. 
Since the stable Auslander-Reiten quivers of $\Lambda$ and $\mathcal{N}$ are isomorphic as valued quivers, 
there must exist integers $e\ge 1$ and $\ell\ge 3$ such that $\mathcal{N}\cong\mathcal{N}(e,\ell)$.
By Theorem \ref{thm:supermain}, for each finitely generated indecomposable non-projective $\mathcal{N}$-module $T$, 
the isomorphism type of $R(\mathcal{N},T)$ is uniquely determined by the distance of $[T]$ to the closest 
boundary in the stable Auslander-Reiten quiver $\Gamma_s(\mathcal{N})$. Since the stable equivalence of Morita 
type preserves these distances in the respective Auslander-Reiten quivers by \cite[Prop. X.1.6]{ars}, part (iii) now follows
from Theorem \ref{thm:supermain} and \cite[Prop. 3.2.6]{blehervelezderived}.
This completes the proof of Theorem \ref{thm:main2}.
\hfill $\Box$

\medskip

\textit{Proof of Corollary $\ref{cor:supermain}$.}
Let $\Lambda$ be a Brauer tree algebra with $e$ edges and an exceptional vertex of multiplicity $m\ge 1$.
By \cite[Thm. 4.2]{Rickard1989}, there exists a derived equivalence between $\Lambda$ and a Brauer tree algebra
whose Brauer tree is a star with $e$ edges and central exceptional vertex of multiplicity $m$.
The latter is Morita equivalent to the symmetric split basic Nakayama algebra $\mathcal{N}(e,me+1)$.
Since Brauer tree algebras are symmetric, it follows by \cite[Cor. 5.5]{Rickard1991} that the derived
equivalence between $\Lambda$ and $\mathcal{N}(e,me+1)$ induces a stable equivalence of Morita type 
between these two algebras. Therefore, Corollary \ref{cor:supermain} follows from Theorem \ref{thm:supermain} 
and \cite[Prop. 3.2.6]{blehervelezderived}. 
\hfill $\Box$

%%%%%%%%%%%%%%%%%%%%%%%%%%%%%%%%%%%%%%%%%%%%%%%%%%%%%%%%%
%% Bibliography
%%%%%%%%%%%%%%%%%%%%%%%%%%%%%%%%%%%%%%%%%%%%%%%%%%%%%%%%%

\end{document}